\documentclass[11pt,reqno]{amsart}
\usepackage{amssymb,amsmath}
\usepackage[bbgreekl]{mathbbol}
\usepackage[toc,page]{appendix}
\usepackage{mathrsfs}
\usepackage[all]{xy}
\newcommand{\cale}{{\mathcal{E}}}

\newcommand{\ideala}{{\mathfrak a}}

\newcommand{\idealp}{{\mathfrak p}}

\newcommand{\idealm}{{\mathfrak m}}

\newcommand{\aut}{{\rm Aut}}
\newcommand{\autsc}{{\mathscr{A}\! ut}}
\newcommand{\fix}{\rm Fix}
\newcommand{\autpol}{{\rm Aut_\K(\pol)}}
\newcommand{\autpolh}{{\rm Aut^h(\pol)}}
\newcommand{\afg}{{\rm Aff(\pol)}}

\newcommand{\jon}{{\rm J(\pol)}}

\newcommand{\sing}{{\rm Sing}}
\newcommand{\bas}{{\rm Bs}}

\newcommand{\bir}{{\rm Bir}}

\newcommand{\pgl}{{\rm PGL}}

\newcommand{\hei}{{\rm ht}}
\newcommand{\len}{{\rm len}}

\newcommand{\Der}{{\rm Der}}

\newcommand{\eig}{{\rm \E}}
\newcommand{\eiga}{{\rm Eing}}

\newcommand{\car}{{\rm char}}

\renewcommand{\P}{{\mathbb P}}

\newcommand{\C}{{\mathbb C}}

\newcommand{\E}{{\mathbb E}}
\newcommand{\F}{{\mathbb F}}

\newcommand{\K}{{\Bbbk}}

\newcommand{\R}{{\mathbb R}}
\newcommand{\Q}{{\mathbb Q}}
\newcommand{\Z}{{\mathbb Z}}

\newcommand{\pol}{{\K[x,y]}}
\newcommand{\fra}{{\K(x,y)}}
\newcommand{\derpol}{{\rm Der}_\K(\pol)}
\newcommand{\deig}{\rm {\mathbb e}}

\newcommand{\tor}{\dashrightarrow}

\newtheorem{THM}{Theorem}

\newtheorem{thm}{Theorem}[section]
\newtheorem{pro}[thm]{Proposition}
\newtheorem{cor}[thm]{Corollary}
\newtheorem{lem}[thm]{Lemma}
\newtheorem{con}[thm]{Conjecture}

\newtheorem{CON}{Conjecture}

\theoremstyle{remark}
\newtheorem{rem}[thm]{Remark}
\newtheorem{rems}[thm]{Remarks}
\newtheorem{exa}[thm]{Example}
\newtheorem{exas}[thm]{Examples}

\begin{document}

\title{On the Automorphism Group of a polynomial differential ring in two variables}
\maketitle
\begin{center}
{\sc Rene Baltazar}\footnote{Research of R. Baltazar  was partially supported by FAPERGS, of Brasil.} and {\sc Ivan Pan}\footnote{Research of I. Pan was partially supported by ANII and PEDECIBA, of Uruguay.}
\end{center}

\begin{abstract}
We consider differential rings of the form $(\Bbbk[x,y],D)$, where $\Bbbk$ is an algebraically closed field of characteristic zero and $D:\Bbbk[x,y]\to \Bbbk[x,y]$ is a $\Bbbk$-derivation. We study the Automorphism Group of such a ring and give criteria for deciding whether that group is an algebraic group. In most cases, from that study we deduce a primary classification of this type of differential ring up to conjugation with a polynomial automorphism.     
  \end{abstract}
\section{introduction}\label{sec_intro}

Let $R=A[x_1\ldots,x_n]$ be a polynomial ring over a commutative ring $A$ and denote by $\aut_A(R)$ the group of $A$-automorphisms of $R$. Denote by $\Der_A(R)$ the $A$-module which consists of all $A$-derivations of $R$.

A \emph{polynomial differential ring} supported on $R$ is a pair $(R,D)$ where $D\in\Der_A(R)$. The \emph{Automorphism Group} of $(R,D)$ is the subgroup of $\aut_A(R)$ consisting of maps commuting with $D$. In other words, $\aut_A(D)$ is the isotropy group relative to the natural action of $\aut_A(R)$ on $\Der_A(R)$ defined by conjugation. 

In the case where $A=\K$ is a field of characteristic zero, an element in $\Der_\K(R)$ carries out with a geometric significance for such an element may be thought of as a vector field on a affine space and then $\aut_\K(R)$ corresponds to the ``polynomial symmetries'' of that vector field.

If in addition $\K$ is the real or complex field, $\R$ or $\C$, respectively, then $D$ may be associated with a (singular) algebraic foliation for which $\aut_\K(D)$ represents a special subgroup of its polynomial symmetries. The geometric meaning of a $\K$-derivation is considerably increased for such fields, and  most of our motivation comes from that context.

On the other hand, the classification problem for polynomial differential rings on $R$ is far from being fixed even in the special case where $A=\K$ and $n=2$; and moreover, little is known about what kind of automorphism group one may expect to appear in that context. The aim of this work is to give a primary (and quite rudimentary) such classification by specifying which kind of automorphism group one may expect to have. We will assume $\K$ to be algebraically closed of characteristic zero.  

Since the ring $R=\K[x,y]$ will remain unchanged throughout all the paper, we will always refer to $\aut(D)$ as the \emph{isotropy} of $D$ instead of speaking about the automorphism group of $(\K[x,y],D)$. All derivations will be assumed to be over $\K$.     

This paper was motivated by a question of Daniel Levcovitz who asked the first author about when the group $\aut(D)$ is algebraic, and it was strongly inspired by the paper \cite{BlSt} by J. Blanc and I. Stampfli which was used several times throughout the present work. As we will see we could answer quite reasonably Levcovitz's question. 

Our results and their corresponding proofs rely on the existence of $D$-stable principal ideals of height 1. More precisely, that a $\K$-derivation $D:\pol\to\pol$ leaves invariant such an ideal is the same as saying there exists $f\in\pol\setminus\K$ such that $D(f)=\lambda f$ for some $\lambda\in\pol$; we will say $f$ is an \emph{eigenvector} of $D$ and $\lambda$ its corresponding \emph{eigenvalue}. The subset of $\pol$ consisting of eigenvectors with null eigenvalue is the so-called \emph{kernel} $\ker D$ of $D$: it is a $\K$-subalgebra of $\pol$.

Denote by $\deig(D)$ the number of $D$-stable reduced principal ideals of height 1 (for more details related to that number see section \ref{sec_dae}).

Since we intend to classify derivations up to conjugation we need to consider eigenvectors, or more generally polynomials, up to apply suitable automorphisms: we will say $f$ is \emph{equivalent} to $g$ if there is $\varphi\in\autpol$ such that $\varphi(g)=f$. In the case where $f$ is equivalent to $x$ we say $f$ is \emph{rectifiable}.

The main results of the paper are the following theorems, where $\K$ denotes, as we have already said, an algebraically closed field of characteristic zero:

\begin{THM}\label{thm_A}
  Let $D$ be a nonzero derivation of $\pol$. Assume $D$ satisfies one of the following properties:

  a) $0<\deig(D)<\infty$ and $D$ admits an eigenvector which is not equivalent to an element in $\K[x]$.

  b) $D$ admits an eigenvector with null eigenvalue which is not equivalent to an element in $\K[x]$; in particular $\deig(D)=\infty$.

  Then $\aut(D)$ is an algebraic group.
\end{THM}

A derivation $D:\pol\to\pol$ is said to be \emph{locally nilpotent} if for any $f\in\pol$ there exists a positive integer $n=n(f)$ such that $D^n(f)=0$.

\begin{THM}\label{thm_B}
  Let $D$ be a nonzero derivation of $\pol$. Then $D$ admits an eigenvector with null eigenvalue which is equivalent to an element in $\K[x]$ if and only if $D$ is conjugate to $b\partial_y$ for some $b\in\pol\setminus\{0\}$. Moreover, in that case we have:
	
	a) If $b\in\K[x]$, then $D$ is locally nilpotent and $\aut(D)$ is not an algebraic group.
	
	b) If $\deg_y b\geq 1$, then $\aut(D)$ is an algebraic group.
\end{THM}

If $D\in\derpol$ is a nonzero derivation, then there exists a polynomial $g$ of maximal degree such that $D=gD_1$, where $D_1\in\derpol$ and $g$ is unique up to multiply by elements in $\K^*$; when $g\in\K^*$ we say $D$ is \emph{irreducible}.

If $h\in\pol$ is an arbitrary polynomial we denote by $\aut(h)$ the subgroup of $\autpol$ whose elements are the automorphisms $\varphi$ such that $\varphi(h)=\alpha h$ for some $\alpha\in\K^*$.


\begin{THM}\label{thm_C}
  Let $D$ be a nonzero derivation of $\pol$; write $D=gD_1$, with $g\in\pol$ and $D_1$ irreducible. Assume $\aut(D)$ is an algebraic group. Then we have:

  $a)$ If $\deig(D)<\infty$ and $D$ admits an irreducible eigenvector which is not equivalent to $x$, then either $\aut(D)$ is finite or it contains a copy of $\K^*$ and there are integers $p,q,\ell\geq 1$, with $p,q$ coprime, such that $D$ is conjugate to $pxa_1\partial_x+qyb_1\partial_y$, where
  \[a_1=\sum_{i=0}^\ell a_{ii}x^{qi}y^{ip},\ b_1=\sum_{i=0}^\ell b_{ii}x^{qi}y^{ip},\]
with $qa_{\ell\ell}+pb_{\ell\ell}\neq 0$.

$b)$ If $\ker D$ contains a non-constant polynomial which is not equivalent to an element in $\K[x]$, then there exists $h\in\ker D$ such that $\aut(D)$ is a closed subgroup of $\aut(h)$. Moreover, one of the following assertions holds:

 \begin{itemize}
  \item[$i$)]$ \aut(h)$ is finite and $\aut(D_1)\subset\aut(h)$ or $\aut(D_1)=\K$.
    \item[$ii$)]  $\aut(h)$ is infinite, $\aut(D_1)$ is isomorphic to $\K^*$ and there are coprime integers $p,q>1$ and $c\in\K^*$ such that $h$ is equivalent to $x^q-cy^p$ and $D_1$ is conjugate to $cpy^{p-1}\partial_x-qx^{q-1}\partial_y$.
\item[$iii)$] $\aut(h)$ is infinite, $\aut(D_1)$ is isomorphic to $\K^*$ and there are coprime integers $p,q> 1$ and $c\in\K^*$ such that $h$ is equivalent to $x^qy^p-c$ and $D_1$ is conjugate to $px\partial_x-qy\partial_y$.

\item[$iv)$] $\aut(h)$ is infinite, $\aut(D_1)$ is isomorphic to $\K^*\rtimes \Z/2\Z$ and there is $c\in\K^*$ such that $h$ is equivalent to $xy-c$ and $D_1$ is conjugate to $x\partial_x-y\partial_y$.

 \item[$v)$] $\aut(h)$ is infinite, $\aut(D_1)\subset\aut(h)$ and $h$ is equivalent to a polynomial of the form $x^{n_1}h_2^{n_2}\cdots h_\ell^{n_\ell}$,  with $\ell\geq 2$ and $h_i$ rectifiable and not belonging to $\K[x]$, $i=2,\ldots,\ell$.
\end{itemize}
\end{THM}

Note that every derivation $D:\pol\to\pol$ extends as a derivation of the \emph{fraction field} $\K(x,y)$ of $\pol$. If $f,g\in\pol$ are irreducible polynomials without nontrivial common factors, then the rational function $f/g$ is null under $D$ if and only if $f$ and $g$ are eigenvectors of equal eigenvalue; in this case $\alpha f+g$ is an eigenvector of equal eigenvalue for any $\alpha\in\K$. A classical result due to J. G. Darboux (Theorem \ref{thm_darboux}) asserts that if $\deig(D)=\infty$, then $D$ annuls non-constant rational functions; the first part of Theorem D is nothing else than a reinterpretation of that result.

For an irreducible polynomial $f$  we define its \emph{genus} to be the geometric genus of the (Zariski closure in the projective plane of the normalization of the) plane curve $f=0$.

\begin{THM}\label{thm_D}
  Let $D$ be a nonzero derivation of $\pol$ such that $\deig(D)=\infty$ and $\ker D=\K$. Then there exist at least two eigenvectors $f,g\in\K[x,y]$, with the same eigenvalue, such that all members of the $1$-parameter family of polynomials $\alpha f+ g$, except for at most a finite number of them, are irreducible and have equal genus. Furthermore, if in addition such a genus is greater or equal to $1$, then $\aut(D)$ is an algebraic group. 
\end{THM}

As an easy consequence of \cite[Thm. A]{CMP}, it follows that the next conjecture (which deals with the last case, i.e. $\deig(D)=0$) is true when $\K=\C$:
\begin{CON}\label{conA}
Let $D$ be a nonzero derivation of $\pol$. If $\deig(D)=0$, then $\aut(D)$ is finite.
\end{CON}

We have organized the paper in a series of five sections, where this introduction is the first of them. More precisely, Section 2 is devoted to giving the preliminary notions concerning the isotropy group associated to a polynomial derivation in two variables, some results about algebraic groups and the so-called ind-groups, and to adapt the results in \cite[\S 3]{BlSt} to a context we will use later.

In the first part of section 3, we treat the case where $D$ admits, up to multiply by nonzero scalars, a finite nonzero number of irreducible eigenvectors and study under what situations either at least one of them is not rectifiable or at least two of the rectifiable ones are algebraically independent over $\K$. Later in section 3, we analyze the case where $\ker D$ contains a rectifiable element; in that section we also give specific results about locally nilpotent derivations (see \S \ref{loc_nil}) as well as about de so-called Shamsuddin derivations (see \S \ref{sham}).

In Section 4 we treat the case where $\ker D\neq\K$.  Theorems \ref{thm_A} and \ref{thm_C} are both consequence of Theorems \ref{thm_finite} and  \ref{thm_ker}, and Theorem \ref{thm_B} is proven after Example \ref{exa_locnil} and is essentially a consequence of Proposition \ref{pro_locnil}.

Finally, in section 5 we consider the remaining case, i.e. $\deig(D)=\infty$ and $\ker D=\K$, and prove Theorem \ref{thm_D}.

\noindent{\bf Acknowledgment}: The second author thanks Alvaro Rittatore for many useful conversations during the elaboration of this paper. He also thanks Thiago Fassarella, Amilcar Pacheco and Jorge Vitorio for their advise concerning some points treated there.

\section{Preparatory material}

In this section we collect all auxiliary results we will need to prove the main theorems stated in the introduction. In order to be self-contained let us first introduce some basic notions for which we will follow \cite[\S 4.1 and 4.2]{Ku}.

An \emph{ind-variety} is a countable union of algebraic varieties $X=\cup_{n=1}^\infty X_n$, over $\K$, such that $X_i$ is a closed subvariety of $X_{i+1}$ for all $i\geq 1$, and where (unlike what it was done in \cite[\S 4.1 and 4.2]{Ku}, see \cite[\S 0]{St}) $X$ is always endowed with the corresponding \emph{inductive topology}, i.e. $F\subset X$ is closed if and only if $F\cap X_n$ is closed for every $n$. We will write $X=\lim_n X_n$ to mean all the preceding data; when all the $X_i'$s are affine their union $X$ itself is said to be affine.

An ind-variety $X=\lim_n X_n$ is an algebraic variety if and only if $X=X_n$ for some $n$. 

A map $\eta: X=\lim X_n\to Y=\lim_m Y_m$ between two ind-varieties is said to be a \emph{morphism} if for every $n$ there is $m=m(n)$ such that $\eta$ induces (by restriction) a morphism of algebraic varieties $X_n\to Y_m$. The morphism $\eta$ is an \emph{isomorphism} if it is bijective and its inverse map is also a morphism.

Note that if $X=\lim_n X_n$, as ind-variety, and $n_1,n_2,\ldots$ is an increasing sequence we may define another ind-variety structure on $X$ by setting $X_m':=\cup_{j\leq m} X_{n_j}$, $X=\lim_m X_m'$. An standard reasoning shows that the two structures on $X$ are isomorphic via the identity map.

By definition, an \emph{ind-subvariety} of an ind-variety $X=\lim_n X_n$ is a closed subset $Z\subset X$ with the natural structure of ind-variety given by $Z_n:=Z\cap X_n$. 

An ind-variety $G=\lim_n G_n$ is said to be an \emph{ind-group} if it is a group such that the map $G\times G\to G$ defined  by $(g,h)\mapsto gh^{-1}$ is a morphism of ind-varieties, where on $G\times G$ we have considered any structure isomorphic to the one given as $(G\times G)_n=G_n\times G_n$; note that $G_n$ is not required to be a group. A morphism of ind-groups is a homomorphism of groups which is a morphism of ind-varieties.

An ind-group  $G=\lim_n G_n$  is an algebraic group if and only if it is an algebraic variety, i.e. when $G=G_n$ for some $n$. An \emph{algebraic subgroup} of $G$ is then a subgroup of $G$ which is an algebraic variety with the ind-subvariety structure induced by $G$.  

\subsection{Generalities about the isotropy group of a derivation}

Let $\K$ be an algebraically closed field of characteristic 0. We denote by $\derpol$ the $\pol$-module of $\K$-derivations in $\pol$; as one knows $\derpol=\pol\partial_x\oplus\pol\partial_y$, where $\partial_x, \partial_y$ are the formal derivatives with respect to $x$ and $y$, respectively. Note that the group $\aut_\K(\pol)$ of polynomial $\K$-linear automorphisms acts on $\derpol$ by conjugation; for short we will refer to elements in that group simply as automorphisms, and to elements in $\derpol$ simply as derivations. If $D:\K[x,y]\to\K[x,y]$ is a derivation we denote by $\aut(D)$ the isotropy subgroup  with respect to the action we have just referred to, i.e. $\aut(D)$ is the subgroup of $\K$-automotphisms $\rho:\pol\to\pol$ such that $\rho D=D\rho$.

The isotropy of $D$ may be seen as a normal subgroup of a bigger subgroup $\autsc(D)$ of $\autpol$ which consists of those automorphisms $\rho$ for which there exists $\alpha\in\K^*$ such that $\rho D=\alpha D\rho$. We have $\autsc(D)/\aut(D)\subset\K^*$.

Analogously, if $h\in\K[x,y]$ is a non-constant polynomial, then we denote $\aut(h)$ the subgroup of $\autpol$ which consists of the automorphisms $\rho$ such that $\rho(h)=\alpha h$ for some $\alpha\in\K^*$. The subset ${\rm Fix}(h)\subset\aut(h)$ whose elements are the automorfisms fixing $h$ is a normal subgroup such that $\aut(h)/{\rm Fix}(h)\subset\K^*$. Note that if $h_0$ is the product of the irreducible factors of $h$ (it is unique up to multiply by an element of $\K^*$), then $\aut(h)\subset\aut(h_0)$ and the first group is a (not necesarily normal) finite index subgroup of the second one.

An automorphism $\rho:\K[x,y]\to\K[x,y]$ is determined by giving $\rho(x)$ and $\rho(y)$, and we will often write $\rho=(f,g)$ by meaning  $\rho(x)=f$ and $\rho(y)=g$, respectively; in this case the \emph{degree} of $\rho$, denoted by $\deg\,\rho$, is the positive integer $\max\{\deg f, \deg g\}$. We have $\autpol=\cup_{d=1}^\infty\aut(\pol)_d$, where $\aut(\pol)_d$ (without the subindex $\K$) is the set of automorphisms of degree $\leq d$. If $\rho,\sigma\in\autpol$, a well known fact says $\deg\,\rho=\deg\,\rho^{-1}$ and $\deg\,\rho\sigma\leq\deg\,\rho\deg\,\sigma$. 
    
Following \cite{Sh} and \cite{Ka} we know that $\aut(\pol)_d$ admits a natural structure of affine algebraic variety, for every $d$, in such a way that $\aut(\pol)_d$ is a closed subvariety of $\aut(\pol)_{d+1}$: roughly speaking, such a structure depends on the coefficients of the couple of polynomials defining elements in $\autpol_d$. In particular one deduces that $\autpol$ admits a structure of affine ind-variety which is compatible with the group structure, i.e. it is an affine ind-group. From now on all topological notions related to $\autpol$ will be referred to the corresponding inductive topology. 

On the other hand, and following again the second reference above, we know that  a subgroup $G$ of $\autpol$ acts on $\pol$ as an algebraic group if and only if it is closed and there is $d$ such that $G\subset \aut(\pol)_d$. Moreover, a closed subgroup of $\autpol$ is algebraic if and only if it is conjugate to a subgroup of either the so-called \emph{affine group} $\afg=\aut(\pol)_1$, or the so-called \emph{de Jonqui\`eres} group $\jon$,  which consists of automorphisms $\rho$ such that $\rho(x)=\alpha x+P(y)$, $\rho(y)=\beta y+\gamma, \alpha,\beta,\gamma\in\K, \alpha\beta\neq 0, P\in\K[y]$; in the second case such a subgroup is then conjugate to a closed subgroup of $J_d:=\jon\cap \autpol_d$, for some $d\geq 1$.

Finally, recall that a derivation $D$ is said to be simple if it does not stabilize nontrivial ideals. If $D$ is simple, then we know $\aut(D)=\{id\}$ (\cite[Thm. 1]{MePa}).

\subsection{Some remarks about ind-groups and algebraic groups}

We start by giving some results concerning ind-groups.

\begin{lem}\label{lem_ind1}
  Let $\eta:H=\lim_n H_n\to G=\lim_m G_m$ be a morphism of ind-groups.

  $a)$ If $H$ is an algebraic group, then $\eta(H)$ is closed; in particular, it is an algebraic subgroup of $G$.
  
  b$)$ If $\eta$ is onto, its  kernel is finite and $G$ is an algebraic group, then $H$ is an algebraic group.
 \end{lem}
 \begin{proof}

  If $H$ is algebraic, then its image under $\eta$ is contained in some $G_m$. Hence $\eta(H)$ is constructible. By an analogous reasoning as (for example) \cite[\S 7.4]{Ham} we show $\eta(H)$ is a closed algebraic group, which proves $a)$.

  To prove $b)$ let us denote by $K$ the kernel of $\eta$ and assume $G$ to be
  an algebraic group.

  Without loss of generality we may suppose  $K\subset H_1$. Moreover, by replacing $H_n$ with $KH_n$, if necessary, we may also suppose every $H_n$ is $K$-stable.

  Now, the quotient $H_n/K$ is an algebraic variety (i.e. a quasi projective variety over $\K$) and $H_n/K\subset H_{n+1}/K$ is closed for $n\geq 1$.

  Since $G=\lim_n H_n/K$ as ind-variety we conclude $H_n/K=H_{n+1}/K$ for $n\gg 0$, hence $H_n\subset H_{n+1}$ for $n\gg 0$. Thus $H$ is algebraic.
\end{proof}

\begin{cor}\label{cor_ind3}
  Let $D\in\derpol$ be a derivation and let $h\in\K[x,y]$. Then we have the following assertions:

  a$)$ $\aut(D)$ is an ind-subgroup of $\autpol$.

  b$)$ $\autsc(D)$ and $ \aut(h)$ are isomorphic to ind-subgroups of $\autpol\times\K^*$.

In particular, any of the subgroups above is an algebraic group if and only if it is contained in $\aut(\pol)_d$ for some $d\geq 1$.
\end{cor}
\begin{proof}
To prove a) it suffices to show that $\aut(D)\cap \aut(\pol)_d$ is closed for every $d\geq 1$. In fact, write $D=a\partial_x+b\partial_y$, with $a,b\in\pol$. An element $\rho=(f,g)\in \autpol_d$ belongs to $\aut(D)$ if and only if 
\begin{equation}\label{eq_isotropy}
a(f,g)=a\partial_x(f)+b\partial_y(f), \ b(f,g)=a\partial_x(g)+b\partial_y(g).
\end{equation}  
These equations may be thought of as a finite number of polynomial equations, depending on the coefficients of $a$ and $b$, that the coefficients of $f$ and $g$ must satisfy. This proves $\aut(D)$ is an ind-subgroup of $\autpol$. The last assertion follows straightforward from \cite[Thm 3.1]{Ka}.

To finish the proof we only consider the assertion relative to $\autsc(D)$ because the other one may be proven analogously.

First note that the natural group homomorphism $\nu:\autsc(D)\to \autpol\times\K^*$ is a morphism of ind-varieties, where its corresponding structure is given as $\autpol\times\K^*=\lim_d (\aut(\pol)_d\times\K^*)$. By an analogous reasoning as above we obtain $\nu(\autsc(D))$ is closed in $\autpol\times\K^*$. On the other hand, the projection map $p_1:\autpol\times\K^*\to \autpol$ is a morphism of ind-groups whose restriction to $\nu(\aut(D))$ defines the inverse of $\nu$, then $\autsc(D)$ is the isomorphic image of an ind-subgroup of $\autpol\times\K^*$, and this proves b).

Now, if $\autsc(D)\subset\aut(\pol)_d$, then its image under $\nu$ is contained in $\aut(\pol)_d\times\K^*$ and then it is an algebraic group. Hence Lemma \ref{lem_ind1} gives $\autsc(D)$ is an algebraic subgroup of $\autpol$. The converse assertion is obtained by applying once \cite[Thm 3.1]{Ka}.

\end{proof}

\begin{lem}\label{lem_ind2}
  Let $H\subset\autpol$ be an ind-subgroup. If $K\subset H$ is an algebraic subgroup of finite index, then $H$ is an algebraic group.
  \end{lem}
 \begin{proof}
There exist $\varphi_1,\ldots,\varphi_\ell\in H$ such that $H=\cup_{i=1}^\ell K\varphi_i$. Since $K$ is algebraic we know there exists $d\geq 1$ such that $K\subset\aut(\pol)_d$. If $d_1$ is the maximum degree of the automorphisms $\varphi_1,\ldots,\varphi_\ell$, then $\deg(\psi)\leq dd_1$ for any $\psi\in H$ which proves this group is algebraic too and completes the proof.   
\end{proof}

The following two results are probably well known exercises we could not find in the literature we consulted.

\begin{lem}\label{lem_exe}
Let $G$ be a connected algebraic subgroup of $\K^*\rtimes \K$. If $\dim G=1$, then $G=\K^*\rtimes\{ 0\}$ or $G=\{1\}\rtimes \K$. 
\end{lem}
\begin{proof}
Since $G$ is connected and has dimension 1 there is an irreducible polynomial $f\in\K[x,x^{-1},y]$ such that $G=(f=0)$; let $m$ be the minimal positive integer such that $g:=x^m f=\sum_{i,j} a_{ij}x^iy^j\in\K[x,y]$, i.e. $(g=0)$ defines the closure of $G$ in $\K\times\K$. Denote by $p_1,p_2:\K\times\K\to\K$ the canonical projections.

  If $(\alpha,\beta)\in G$, then
  \begin{equation}\label{eq_character}
    g(\alpha x,\alpha y+\beta)=\alpha^m\lambda(\alpha,\beta) g(x,y)\end{equation} where $\lambda$ is a character of $G$.

First we prove there are no $r,s>0$ such that  $a_{rs}\neq 0$. In fact, assume, by contradiction such $r,s$ exist and let $\ell\geq s$ be the biggest integer such that $a_{r\ell}\neq 0$. Hence both projections $p_1(G)$ and $p_2(G)$ are dense in $\K$. Note that $\dim\,G=1$ implies $p_i(G)$ is dense for some $i\in\{1,2\}$, and that if $p_i(G)$ is not dense, then necessarily $p_i(G)=\{e\}$ where $e=2-i\in\K$.     

Now, on the one hand, by comparing the coefficients of $x^ry^\ell$ in both sides of (\ref{eq_character}) we deduce $\alpha^m\lambda(\alpha,\beta)=\alpha^{r+\ell}$. Then $\lambda=\alpha^{r+\ell-m}$ is a character of the tore $\K^*$. On the other hand, by expanding the left side of (\ref{eq_character}) as a polynomial in $\beta$, and taking into account that $p_2(G)$ is dense in $\K$, we get a contradiction.

Therefore $g(x,y)=g_1(x)+g_2(y)$. If $p_2(G)$ is dense in $\K$, then we deduce $g=ax+b$, with $a\in\K^*, b\in\K$.  Hence the only possibility is $-b/a=1$, i.e. $G=\{1\}\rtimes \K$.

If $p_2(G)$ is not dense in $\K$, then $p_2(G)=\{0\}$, $g=ay$, with $a\in\K^*$, and $G=\K^*\rtimes 0$.
\end{proof}

\begin{lem}\label{lem_agroup2}
An algebraic group $($over $\K=\overline{\K}$ of characteristic zero$)$ of positive dimension is not a torsion group. In particular, such an algebraic group admits an element of infinite order. 
\end{lem}
\begin{proof}
  Without loss of generality we may assume that $G$ is connected. Chevalley's Theorem \cite{Ch} asserts there is an exact sequence (see also \cite{BSU})
  \[\xymatrix{1\ar@{->}[r]&G_{af\!f}\ar@{->}[r]&G\ar@{->}[r]&A\ar@{->}[r]&0},\]
  where $G_{af\!f}$ is affine and $A$ abelian. Since $A$ has positive rank under the present hypothesis (\cite[Thm. 10.1]{FrJa}) it suffices to consider the case $G=G_{af\!f}$ is affine.

 If the unipotent radical of $G$ has positive dimension, then it contains a copy of $\K$ and we are done. Otherwise $G$ is reductive and then it contains a maximal torus of positive dimension. In the last case $G$ then contains a copy of $\K^*$, which completes the proof.
\end{proof}

If $p,q\geq 1$ are coprime positive integers we consider $\K^*$ as a subgroup of $\autpol$ in two different forms via the two actions of it on $\pol$ given, respectively, by
\[t\cdot x=t^px, t\cdot y=t^qy\]
and 
\[t\cdot x=t^px, t\cdot y=t^{-q}y;\]
we denote by $G_{p,q}$ and $G_{p,-q}$ the corresponding subgroups in $\autpol$.

Lemma \ref{lem_agroup2} may be used together with the next one:

\begin{lem}\label{lem_h}
  Let $h$ be a polynomial which is not equivalent to an element in $\K[x]$. If $\aut(h)$ is not a torsion group, then $\aut(h)$ is an algebraic group and one of the following assertions holds:
\begin{itemize}
\item[$a)$] there are coprime integers $p,q> 1$ and $c\in\K^*$ such that $h$ is equivalent to $(x^q-cy^p)^n$, for some $n\geq 1$, and $\aut(h)$ is isomorphic to $G_{p,q}$.
\item[$b)$] there are coprime integers $p,q\geq 1$ and $c\in\K^*$ such that $h$ is equivalent to $(x^qy^p-c)^n$, for some $n\geq 1$, and $\aut(h)$ is isomorphic to $G_{p,-q}$ or $G_{p,-q}\rtimes \Z/2\Z$ depending on $(p,q)\neq (1,1)$ or $(p,q)=(1,1)$, respectively.
  \item[$c$)] $h$ is equivalent to a polynomial of the form $x^{n_1}h_2^{n_2}\cdots h_\ell^{n_\ell}$, with $\ell>1$, such that $h_2,\ldots,h_\ell$ are rectifiable and at least one of them depends on $y$. 
\end{itemize}
\end{lem}
\begin{proof}
Let $\psi\in\aut(h)$ be an element of infinite order. We know $\psi^m$ stabilizes all irreducible components of $h$. From \cite[Thm2]{BlSt} it follows that up to conjugate $\aut(h)$ with an element in $\autpol$ we may assume $h$ to admit an irreducible decomposition of the form $h_1^{n_1}\cdots h_\ell^{n_\ell}$ in such a way one of the following situations occurs:

   $(i)$  there are $p,q\geq 1$ and $c\in\K^*$ as in the statements $a)$ or $b)$ above, where $h_1$ is of the form either $x^qy^p-c$, with $p,q>1$, or $x^qy^p-c$.

   $(ii)$ $h_1=x$.

   Suppose we are in the situation $(i)$. In the first case there $\aut(h_1)$ is $G_{p,q}$ and in the second one it is $G_{p,-q}$ or $G_{p,-q}\rtimes \Z/2\Z$ depending on $(p,q)\neq (1,1)$ or $(p,q)= (1,1)$.

   By replacing $m$ with $2m$, if necessary, we may assume $\psi^m=(t^px,t^qy)$ or $\psi^m=(t^px,t^{-q}y)$, respectively, for some $t\in\K^*$ such that $t^n\neq 1$ for any $n\geq 1$. We conclude that all irreducible factor of $h$ has the same form: indeed, for example let us assume $h_2=\sum a_{ij}x^iy^j$. Then there is $\alpha\in\K^*$ such that $\sum a_{ij}(t^{pi\pm qj}-\alpha)x^iy^j$, where the sign plus or minus correspond to elements in $G_{p,q}$ or $G_{p,-q}$, respectively. Hence $a_{ij}\neq 0$ implies $t^{pi\pm qj}-\alpha=0$. In the ``plus case''  there are at most two different pairs $(i,j),(i',j')$ such that $a_{ij},a_{i'j'}\neq 0$. We get $i'=i+nq, j'=j-np$. Then the only possibility is $(i,j)=(q,0), (i',j')=(0,p)$ and $t^{pq}=\alpha$. The ``minus case'' is similar.

    Then there is a positive integer $n\geq 1$ such that $h$ is as in the statements  $a)$ or $b)$ above; in particular $\aut(h)$ is as it is asserted there. Note that $\aut(h)$ is an algebraic group in these cases.

    Now suppose we are in the situation $(ii)$. Since $\aut(h/x^{n_1})$ has also an element of infinite order we deduce $h/x^{n_1}$ is as in $(i)$ or $(ii)$ above. As we have seen $(i)$ implies all irreducible component of $h$ has the same form hence $h$ is a product of rectifiable elements. By hypothesis $h/x^{n_1}$ depends effectively on $y$. 

    In order to complete the proof it remains to prove $\aut(h)$ is an algebraic group also in this last case. Let $\varphi\in\aut(h)$ be and arbitrary element.

    We have an exact sequence of groups $\xymatrix{1\ar@{->}[r]&K\ar@{->}[r]&\aut(h)\ar@{->}[r]&F}$, where $F$ is the permutation group associated to the set of irreducible components of $h$. Then $K\subset \aut(x)\cap\aut(h_2)\cap\cdots\cap\aut(h_\ell)$. Assume $h_2=\sum_{i=0}^na_i y^i$, with $n\geq 1$, $a_i\in\K[x]$ for all $i$ and $a_n\neq 0$.

    Notice that $\varphi\in\aut(x)$ implies $\varphi$ is an automorphism of the form $(\alpha x,\beta y+P(x))$. On the other hand, $(\alpha x,\beta y+P(x))\in\aut(h_2)$ implies
    \[a_n(\alpha x)P(x)^n+\cdots+a_{1}(\alpha x)P(x)+ a_0(\alpha x)=\gamma a_0(x)\]
    for some $\gamma\in\K^*$. If $P\neq 0$ we deduce $\deg P$ is bounded by $\max\{\deg a_i; a_i\neq 0, i=0,\ldots,n\}$. Hence $K$ is an algebraic group, so $\aut(h)$ does by Corollary \ref{cor_ind3}. 

  \end{proof}

  \begin{exas}\label{exas_rs}
    $a)$ If $h=x^ry^s$, then $\aut(h)$ is isomorphic to  $\K^*\times\K^*$ via the action on $\autpol$ given as $(\alpha,\beta)\mapsto (\alpha x,\beta y)$.

    $b)$ If $h=x^ry^s(y-x^n)$, then $\aut(h)$ is isomorphic to $\K^*$ via $\alpha\mapsto (\alpha x,\alpha^n y)$.
  \end{exas}

\subsection{Automorphisms of the plane preserving a pencil}\label{subsec_pencil}

In this subsection we rewrite some results in \cite[\S 3]{BlSt} in order to apply them to the case of automorphisms which preserve a dimension 1 lineal system of curves, that is a \emph{pencil}, when $\K$ is algebraically closed of characteristic zero. More precisely, and following the authors of {\em loc. \!cit.}, we consider \emph{natural completions} of $\K^2$ of the form $(X,B_X)$, where $X$ is the projective plane obtained by adding a projective line $B_{\P^2}=L_{\infty}$ to $\K^2=\P^2\setminus L_{\infty}$, or $X=\F_n$ is a Nagata-Hirzebruch surface, $B_X=L_n\cup E_n$ and $\K^2=\F_n\setminus B_{\F_n}$, where $L_n$ is a fiber of the fibration $\F_n\to\P^1$ and $E_n$ the $(-n)$-curve of $\F_n$; here $n\geq 1$.

We have tried to be as self-contained as possible, but for some details, and not to be excessively repetitive, we will refer the reader to consult the paper above. Moreover, except for Lemma \ref{lem_blst1}, the proofs of our results below are mere adaptations of some of the proofs therein.  

A birational map (respectively, an isomorphism) between two natural completions $\phi: (X,B)\tor (X',B')$ is a birational map $X\tor X'$ (respectively, an isomorphism $X\to X'$) which induces an isomorphism $X\setminus B\to X'\setminus B'$.

An \emph{elementary link} between two natural completions is one of the following birational maps: the blow-up $\P^2\tor \F_1$ of a point in $L_\infty$, the contraction $\F_1\to \P^2$ of the $(-1)$-curve $E_1$, a map $\F_n\tor \F_{m}$ obtained by first blowing up a point in $L_n$ and then contracting the strict transform of $L_n$; note that $m=n+1$ if the blown up point belongs to $E_n$ and $m=n-1$ otherwise. One refers to these three types of elementary links as links of type I, III and II, respectively. We know that a birational map $\varphi: (X,B)\tor (X',B')$ which is not an isomorphism decomposes into the product of a minimal number of elementary links $\varphi=\varphi_m\cdots \varphi_1$ such that $\varphi_{i+1}\varphi_{i}$ is never an isomorphism unless $\varphi$ itself does. We say that product is a \emph{reduced decomposition} of $\varphi$ into elementary links and call $m$ the \emph{length} of $\varphi$; we denote $\len(\varphi)$ that length and consider $\len(\varphi)$ to be 0 when $\varphi$ is an automorphism (see \cite[Prop. 2.10]{BlSt}).  

Let $\Lambda$ be a pencil of curves on $X$ without fixed part in $B$. Recall that the \emph{base locus} $\bas(\Lambda)$ of $\Lambda$ consists of the base points of $\Lambda$, i.e. the points which belong to all its members.

If $\varphi: (X,B)\tor (X',B')$ is a birational map we denote by $\varphi_*\Lambda$ the \emph{strict transform} of $\Lambda$ under $\varphi$, i.e $\varphi_*\Lambda$ is the pencil whose general members are the strict transform of general members of $\Lambda$. We denote by $\bir((X,B),\Lambda)$ (respectively, $\aut((X,B),\Lambda)$) the group of birational maps (respectively, automorphims) of $(X,B)$ preserving $\Lambda$.

If $A\subset X$ is an irreducible curve in $X$ which is not a fixed part of $\Lambda$ and $T\subset A$ is a finite subset, possibly empty, a general member in $\Lambda$ intersects $A\setminus T$ in a constant number of points taking into account multiplicities; we denote $(\Lambda\cdot (A\setminus T))$ that constant number. If $p\in X$, the intersection multiplicity $(\Lambda\cdot A)_p$ of $\Lambda$ and $A$ at $p$ is by definition the corresponding intersection multiplicity $(C\cdot A)_p$ of $A$ and a general member $C$ of $\Lambda$; clearly $(\Lambda\cdot A)_p\neq 0$ if and only if $p \in\bas(\Lambda)$. We will say $\Lambda$ \emph{intersects transversely} $A\setminus T$ if a general member $C$ of $\Lambda$ verifies $(C\cdot A)_p\leq 1$ for all $p\in A\setminus T$.

Finally, we recall the so-called height $\hei_C(p)$ of a curve $C\subset X$ at a point $p\in C$ given in \cite[Def. 3,3]{BlSt} and introduce the \emph{height} of $\Lambda$ at a base point $p\in\bas(\Lambda)$ to be the height $\hei_\Lambda(p)$ of a general member in $\Lambda$.

From now on we assume (by simplicity) $\Lambda$ has no fixed part, i.e., $\bas(\Lambda)$ is finite. On the other hand, any curve in the surfaces $\P^2$ or $\F_n$ ($n\geq 1$) is assumed to be projective.  

\begin{lem}\label{lem_blst1}
If $L\subset X$ is an irreducible smooth curve, then $\Lambda$ intersects transversely $L\setminus\bas(\Lambda)$.
\end{lem}
\begin{proof}
  The pencil $\Lambda$ defines a rational map $\lambda:X\tor\P^1$ such that $\bas(\lambda)=\bas(\Lambda)$ and the general members of $\Lambda$ correspond to the closure of its general fibers. Assume, by contradiction, for a general point $p\in L\setminus\bas(\Lambda)$ there is a general member $C$ of $\Lambda$ such that $(C\cdot L)_p>1$. By Bertini's Theorem the curve $C$ is smooth at $p$, hence $C$ and $L$ have the same tangent line at $p$.

  Now, $\lambda$ restricts to $L$ as a morphism $L\to\P^1$. If that morphism is constant, then $L$ is part of a fiber of $\lambda$ and then $(C\cdot L)_p=0$. Otherwise, $p$ is one of its ramification points, whose number if finite, which gives a contradiction and proves the lemma.   
\end{proof}

\begin{lem}\label{lem_blst2}
  Let $\varphi\in\bir((X,B),\Lambda)$ be a birational map which is not an isomorphism. Consider a reduced decomposition $\varphi_m\cdots \varphi_1$ of $\varphi$ into elementary links and assume $\varphi$ is of type I or II; denote by $q$ the base point of $\varphi_1$. Then one of the following holds

  a$)$ $(\Lambda\cdot (L_X\setminus\{q\}))\leq 1$;

  b$)$ $(\Lambda\cdot (L_X\setminus\{q\}))> 1$, all member of $\Lambda$ is singular at $q$ and this point is the (unique) one in $L_X$ for which the height of $\Lambda$ is maximal; in particular $q\in\bas(\Lambda)$. 
\end{lem}
\begin{proof}
  Let $C\in\Lambda$ be a general member. Hence $C':=\varphi_*C$ is a general member of $\Lambda$. Notice that $\varphi_m^{-1}$ admits a base point, $p$ say, since $\varphi_1$ is not of type III. If $(\Lambda\cdot (L_X\setminus\{q\}))> 1$, then Proposition \cite[Pro. 3.4]{BlSt} implies $\hei_{C'}(p)>\hei_C(r)$ for any $r\in (L_X\setminus\{q\})\cap C$. In particular $p$ is a singular point of any element in $\Lambda$, hence $p\in\bas(\Lambda)$. Thus $p=q$ and the assertion follows. 
\end{proof}

\begin{pro}\label{pro_blst3}
  Let $(X,B)$ be a natural completion of $\K^2$ and let $\Lambda$ be a pencil of curves on $X$ (without fixed part). Then there exists a birational map $\varphi:(X,B)\tor (X',B')$ such that one of the following holds

  a$)$ $\bir((X',B'),\varphi_*\Lambda)\subset \aut(X',B')$;

  b$)$ $\varphi_*\Lambda$ intersects transversely $L_{X'}$. 
\end{pro}
\begin{proof}
  Set $G:=\bir((X,B),\Lambda)$. Let us assume that a) and b) do not hold for $\varphi=id$ and $X'=X$. Then Lemma \ref{lem_blst1} implies there is $p\in L_X\cap\bas(\Lambda)$ such that $(C\cdot L_{X})_p> 1$ for any $C\in\Lambda$. First of all, and as a first step of the proof, we will show there is an elementary link $\varphi:(X,B)\tor (X_1,B_1)$ such that $\len(\varphi g\varphi^{-1})=0$ if $\len(g)=0$ and $\len(\varphi g\varphi^{-1})<\len(g)$ otherwise. 

  If all $g\in G\setminus\aut(X,B)$ admits a reduced decomposition into elementary links $g_m\cdots g_1$, where $g_1$ is a link of type III, that is $X=\F^2$ and $g_1:\F_1\to\P^2$ is the map which contracts $E_1$ onto a point, then $g_m$ is a link of type $I$; note that $g_1$ is unique up to compose with an automorphism of $\P^2$. Since $\varphi\aut(X,B)\varphi^{-1}\subset \aut(\P^2,L_\infty)$ the map $\varphi:=g_1$ works in this case.

  On the other hand, suppose there is $g\in B\setminus\aut(X,B)$ which admits a reduced decomposition into elementary links $g_m\cdots g_1$, where $g_1:(X,B)\tor (X',B')$ is a link of type I or II; denote by $q$ the base point of $g_1$. If $q\neq p$, Lemma \ref{lem_blst2} implies that all member of $\Lambda$ is singular at $q$ and $\hei_\Lambda(q)$ is the biggest height of a general member of $\Lambda$, and $q\in\bas(\Lambda)$. If $q=p$ then we already know $q\in\bas(\Lambda)$ and that lemma assures $\hei_\Lambda(q)$ is the biggest height as before. In both cases $(\Lambda\cdot L_X)_q>1$.

 Now, take $f\in G$. If $f\in\aut(X,B)$, then $f(q)=q$ and $g_1fg_1^{-1}\in\aut(X',B')$ in which case $\varphi=g_1$ is as required.

  If $f\not\in\aut(X,B)$, then we first note that a reduced decomposition of it into elementary links can not start with a link of type III: indeed, if such a decomposition starts with $f_1:\F_1\to\P^2$, then $f_1:\F_1\tor\F_2$ and $L_1\cap E_1=\{q\}$. Part (i) of Proposition \cite[Pro. 3.4]{BlSt} applied to $\psi:=hf_1^{-1}$ and to a general member $C$ of $\Lambda$ implies there is a point $q'\in L_X$ such that $\hei_{\psi_*C}(q')>\hei_\Lambda(q)$, which gives a contradiction. Therefore a reduced decomposition of $f$ starts with a link $f_1$ of type I or II, and Lemma \ref{lem_blst2} then gives its base point is precisely $q$. So $\varphi=f_1$ is as required, and that finishes the first step of the proof.

Let us denote $\Lambda_1:=\varphi_*\Lambda$.  If a) and b) do not hold relatively to $(X_1,B_1)$, $\Lambda_1$  and $G_1=\varphi G\varphi^{-1}$, then we apply once the first step to produce an elementary link $\varphi_1:(X_1,B_1)\tor (X_2,B_2)$; and so on. Hence we may construct a sequence of elementary links
\[\xymatrix{(X,B)\ar@{-->}[r]^\varphi&(X_1,B_1)\ar@{-->}[r]^{\varphi_1}&(X_2,B_2)\ar@{-->}[r]^{\varphi_2}&\cdots},\]
a sequence of groups $G,G_1,G_2,\ldots$ and a sequence of pencils $\Lambda,\Lambda_1,\Lambda_2,\ldots$, such that $G_i$ and $\Lambda_i$ do not verify neither a) nor b).

To finish the proof it suffices to show that there is $j\geq 2$ such that $G_j\subset \aut(X_j, B_j)$ or $\Lambda_j$ intersects transversely $L_{X_j}$. In fact, if that does not occur for $j=i-1$, by construction either $\varphi_{i}$ is of type III, $X_i=\P^2$ and $\varphi_i G_j\varphi^{-1}\subset\aut(\P^2,L_\infty)$ or $\varphi$ is of type I or II and its base point belongs to $\bas(\Lambda_j)$. In the last case we deduce that the number of base points (taking into account proper and infinitely near base points) of $\Lambda_i$ in (and over) $L_{X_i}$ is less than the number of base points of $\Lambda_j$ in (and over) $L_{X_j}$. Hence the cardinal of $\bas(\Lambda)$ is finite we obtain the result.
\end{proof}

\begin{lem}\label{lem_blst3}
  Let us suppose $(X,B)=(\P^2,L_\infty)$ and let $\Lambda$ be a pencil which intersects transversely $L_\infty$ such that $(\Lambda\cdot L_\infty)=2$. Then $\bir((\P^2,L_\infty),\Lambda)\subset\aut(\P^2,L_\infty)$.
\end{lem}

\begin{proof}
The proof of \cite[Lem. 3.7]{BlSt} may be readily adapted to our situation.
\end{proof}

A pencil $\Lambda$ on $\P^2$ is said to be composed by lines if there is a point $p\in\P^2$ and a positive integer $\ell\geq 2$ such that the general member of $\Lambda$ is the union of $\ell$ lines passing through $p$; in this case $\bas(\Lambda)=\{p\}$.

\begin{cor}\label{cor_blst4}
Let $(X,B)$ be a natural completion of $\K^2$ and let $\Lambda$ be a pencil of curves on $X$ (without fixed part) which intersects transversely $L_X$. Then there exists a birational map $\varphi:(X,B)\tor (X',B')$ such that one of the following holds

  a$)$ $\bir((X',B'),\varphi_*\Lambda)\subset \aut(X',B')$;

  b$)$ $X'=\P^2$ and $\varphi_*\Lambda$ is a pencil composed by lines through a point in $L_{\infty}$.
\end{cor}  

\begin{proof}
  Assume there exists $g\in\bir((X,B),\Lambda)\setminus \aut(X,B)$, and write $g=\varphi_{m}\cdots\varphi_{1}$ as a reduced decomposition into elementary links. The proof is then consequence of considering the following cases:

 (1) Suppose $X=\P^2$. By Lemma \ref{lem_blst2} a general member of $\Lambda$ intersects $L_X$ in at most two points, hence $\Lambda$ is a pencil of lines through a point in $L_X$ or it is a pencil of conics whose general member is not tangent to $L_X$. In the first case we are done and in the second one the assertion follows from Lemma \ref{lem_blst3}.

 (2) Suppose $X\not\simeq\P^2$ and there is $1<i<m$ such that $\varphi_i:\F_1\to\P^2$ is a link of type III contracting $E_1$ onto a point $q\in L_\infty$. If $\Lambda':=(\varphi_{i}\cdots\varphi_{1})_*\Lambda$ intersects transversely $L_\infty$ we choose $\varphi$ to be $\varphi_{i}\cdots\varphi_{1}$ and the assertion follows from (1). Otherwise, accordingly to Lemma \ref{lem_blst1} there is $p\in\bas(\Lambda')$ such that $(\Lambda'\cdot L_\infty)_p>1$; note that $\varphi_{i+1}$ is of type I. If $C'$ is a general member of $\Lambda'$, then \cite[Pro. 3.4]{BlSt} implies either $(\varphi_m\cdots\varphi_{i+1})_*(C')$ is singular at a point of $L_X$ or $p$ is the base point of $\varphi_{i+1}$. The first occurrence is not possible by the hypotheses on $\Lambda$ and the second one contradicts that $\varphi_m,\ldots,\varphi_{i+1}$ is reduced (\cite[Lem. 2.14]{BlSt}).

 (3) $X=\F_n$ for some $n\geq 1$ and all $\varphi_i$ is of type II. By Lemma \ref{lem_blst2} a general member $C$ of $\Lambda$ is disjoint from $L_X$ or intersects transversely it at one or two points. In the first case $C$ is a fiber of $\F_n\to\P^1$, hence one may conjugate with a birational map $\phi:(\F_n,L_n\cup E_n) \to (\P^2,L_\infty)$ in such a way that $\phi_*\Lambda$ is composed by lines, as required. From now on we assume $C\cap L_X\neq\emptyset$.

Finally, up to conjugate, if necessary, with a product of $\ell\leq 1$ elementary links of type II $\F_n\tor\F_{n+\ell}$, we may additionally assume $\bas(\Lambda)\cap L_n\cap E_n=\emptyset$; note that $1\leq \# (\bas(\Lambda)\cap L_X)\leq 2$. The final part of the proof in \cite[Pro. 3.8]{BlSt} works to show that under these assumptions we eventually get a contradiction, which completes our proof.
\end{proof}

\section{Derivations admitting eigenvectors}\label{sec_dae}
We recall that a derivation $D:\pol\to\pol$ extends to a derivation in the field of fractions $\fra$ of $\pol$, which we will still denote by $D$. 

Following the terminology used by van den Essen in \cite{vdE} we say that a non-constant element $f\in\pol$ is an \emph{eigenvector} of $D$ if there is $\lambda\in\pol$ such that $D(f)=\lambda f$; in that case $\lambda$ is called the \emph{eigenvalue} of $f$. Note that if $f,g\in \pol$ are eigenvectors of $D$, $c\in\K^*$ and $n\geq 2$, then $fg$, $cf$, $f^n$ and every irreducible component of $f$ are also eigenvectors.

We denote by $\eiga(D)\subset\pol$ the subset made up of the eigenvectors of $D$ and by $\eig(D)$ the subset in the projective space $\P(\pol)$ which consists of the classes $[f]$ of elements $f\in\eiga(D)$ which are \emph{square free}, i.e. $[f]\in\eig(D)$ implies there are no $g\in\eiga(D)$ and $n\geq 2$ such that $g^n$ divides $f$; we want to allow $\eig(D)$ to be finite. Note that the number $\deig(D)$ in the statement of Theorem \ref{thm_A} is precisely the number of elements of $\eig(D)$. One readily shows that if $\eig(D)$ is finite, then $\deig(D)=2^{\ell}-1$, where $\ell$ is the number of principal $D$-stable ideals of height 1 which are prime.

The \emph{kernel} of $D$ is the subring $\ker D=\{h\in\pol; D(h)=0\}$ whose non-constant elements are of course contained in $\eiga(D)$. Note that $\ker D\neq \K$ implies $\deig(D)=\infty$.

Elements in $\eiga(D)$ are often called \emph{Darboux polynomials} of $D$ (see for example \cite{Now1}).  

The following theorem is due to Jean G. Darboux and is well known for $\K=\R$ or $\K=\C$ (see \cite{Da} or \cite[Pro. 3.6.8]{Jo}). In a private communication to the second author of this paper Thiago Fassarella pointed him out that one of the known proofs of that result, using differential forms, may be adapted to the case of an abstract field of characteristic 0. Following the idea of Fassarella we wrote a proof of that nice and useful result in our context of derivations, and which is presented in the appendix, at the end of this paper.

\begin{thm}\label{darboux}$($Darboux, 1879$)$\label{thm_darboux} 
Let $D=a(x,y)\partial_x+b(x,y)\partial_y$ be a derivation and let $d$ be an integer such that $d\geq \deg  a(x,y), \deg  b(x,y)$. If $\eig(D)$ contains at least $2+d(d+1)/2$ irreducible elements, then there is $h \in \K(x,y) \setminus \K$ such that $D(h)=0$.
\end{thm}    

An element $h\in \K(x,y) \setminus \K$ such that $D(h)=0$ is said to be a \emph{rational first integral} for $D$.

\begin{cor}\label{cor_darboux}
	The following assertions are equivalent for a derivation $D\in\derpol$:
	
	$a)$ $\eig(D)$ contains an infinite number of elements;
	
	$b)$ there exists $h\in\fra\setminus\K$ such that $D(h)=0$;
	
	$c)$  there exist $f,g\in\eiga(D)$ with the same eigenvalue and such that $[f]\neq[g]$.
	
In particular, if $\ker D$ contains $\K$ strictly, then the equivalent assertions above are verified.
\end{cor}
\begin{proof}
	Note that if $h=f/g$ is a rational function with $f,g\in\pol$ coprime, then $D(h)=0$ is equivalent to $D(f)/f=D(g)/g$ from which it follows $b)\Leftrightarrow c)$. On the other hand, Darboux Theorem gives $a)\Rightarrow b)$. In order to prove the equivalence among assertions $a),b)$ and $c)$ it suffices to prove 
	$c)\Rightarrow a)$: in fact, $D(f)=\lambda f, D(g)=\lambda g$ implies $\alpha f+\beta g\in\eiga(D)$ for all $\alpha,\beta\in\K$.
	
	To finish the proof we note that if $h\in\pol\setminus\K$ belongs to $\ker D$, then the entire sub $\K$-algebra $\K[h]$ does, which implies $\K[h]\subset \eiga(D)$ and then $\#\,\eig(D)=\infty$.
\end{proof}

\begin{rem}\label{rem_integral}
  According to \cite[Thms. 10.1.1 and 10.1.2]{Now1} we know that $\ker D$ may be $\K$ even if $\eig(D)$ is infinite. 
\end{rem}

\begin{exa}\label{exa_ydy}
Consider the derivation $D=y\partial_y:\K[x_1,\ldots,x_n,y]\to\K[x_1,\ldots,x_n,y]$. Then $f\in\eiga(D)$ with eigenvalue $\lambda$ if and only if there exists a non-negative integer $\ell$ such that $f=ay^\ell$ for an $a\in\K[x_1,\ldots,x_n]$, in which case $\lambda=\ell$.
\end{exa}

Let $h\in\pol$. We recall $h$ is said to be equivalent to $h'\in\K[x,y]$ if there is $\varphi\in\autpol$ such that $\varphi(h')=h$, and if $h'=x$ we say $h$ is rectifiable; note that a polynomial which is rectifiable necessarily is (non constant and) irreducible.

\begin{exa}
Consider the derivation $D=c x\partial_x$, where $c\in\K^*$. Then $x$ is a rectifiable eigenvector of $D$. An automorphism $\varphi=(f,g)$ belongs to the isotropy of $D$ if and only if $c xf_x(x,y)=cf(x,y)$ and  $cxg_x=0$; hence $f=xf_1(y)$ and $g=g_1(y)$ with $f_1,g_1\in\K[y]$. Taking into account that the Jacobian of $\varphi$ is constant we deduce 
\[\aut(D)=\{(\alpha x, \beta y+\gamma); \alpha\beta\neq 0\}\simeq (\K^*\times\K^*)\rtimes\K.\]
\end{exa}

\begin{exa}\label{exa_finite}
Consider the derivation $D=x\partial_x+\partial_y$. We have $\eig(D)=\{[x]\}$. Indeed, if $f\in\pol$, then $D(f)=\lambda f$ implies $\lambda\in\K$. The assertion readily follows by writing $f=\sum_{i=0}^n a_iy^i$, with $a_i\in\K[x]$ and expressing that equality in terms of the $a_i$'s.
\end{exa}

\begin{thm}\label{thm_finite}
Let $D:\K[x,y]\to\K[x,y]$ be a derivation such that $0<\#\eig(D)<\infty$. If $\eig(D)$ contains an element which is not equivalent to an element in $\K[x]$, then $\aut(D)$ is an algebraic group and it holds exactly one of the following assertions

$a)$ $\aut(D)$ is finite.

$b)$  There are integers $p,q,\ell\geq 1$, with $p,q$ coprime, and $\varphi\in\autpol$ such that:
\begin{itemize} 
\item[$i$)] $\varphi D\varphi^{-1}=pxa_1\partial_x+qyb_1\partial_y$ where \[a_1=\sum_{i=0}^\ell a_{ii}x^{qi}y^{ip},\ b_1=\sum_{i=0}^\ell b_{ii}x^{qi}y^{ip},\]
with $qa_{\ell\ell}+pb_{\ell\ell}\neq 0$.

\item[$ii$)] The polynomial $x^qy^p-c$ is an eigenvector of $\varphi D\varphi^{-1}$ if and only if $c$ is either zero or a solution of the equation  $\sum_{i=0}^\ell (qa_{ii}+pb_{ii})z^i=0$.

\item[$iii$)] $\aut(\varphi D\varphi^{-1})$ contains $G_{p,-q}$ or $G_{p,-q}\rtimes \Z/2\Z$, where $\Z/2\Z$ is generated by the involution $(x,y)\mapsto (y,x)$, depending on $(p,q)\neq (1,1)$ or $(p,q)= (1,1)$.
\end{itemize}
In particular, if $\aut(D)$ is infinite, then $\eig(D)$ contains at least two rectifiable elements which are algebraically independent over $\K$.
\end{thm}

\begin{proof} There is a group homomorphism from $\aut(D)$ to the permutations group of $\eig(D)$. The kernel $K$ of this homomorphism is the normal subgroup of $\aut(D)$ whose elements are those which stabilize the curves of the form $(h=0)$ for any $h\in\eiga(D)$.

  \noindent{\it Claim}. The subgroup $K$ is an algebraic group. Indeed, if $h\in\eig(D)$ is not equivalent to an element in $\K[x]$, then the product of their irreducible components does, hence we may assume such a $h$ to be a reduced polynomial. The assertion then follows from \cite[Thm. 1]{BlSt}.

The claim above together with Lemma \ref{lem_ind2} shows $\aut(D)$ is an algebraic group. Hence the first assertion is proved.

To prove the rest of the theorem assume $\aut(D)$ is not finite. Lemma \ref{lem_agroup2} implies there is an element $\rho\in\aut(D)$ which has infinite order. Therefore we have a positive integer $N$ such that $\rho^N\in K$, i.e.  $\rho^N$  fixes all points in $\eig(D)$. 

Write $D=a\partial_x+b\partial_y$ and note $ab\neq 0$ since otherwise $\eig(D)$ would not be finite. Take an irreducible eigenvector $h\in\pol\setminus\K$ which is not equivalent to an element in $\K[x]$. Up to conjugate with an element $\varphi\in\autpol$ we may assume $D=\varphi D\varphi^{-1}$ and $h$ as in $a)$ or $b)$ of Lemma \ref{lem_h}, with $n=1$. We treat each case separately:

In case $a)$ there are coprime integers $p,q>1$ and $c,t\in\K^*$ such that $h=x^q-cy^p$ and  $\rho^N=(t^px,t^qy)$. Equations (\ref{eq_isotropy}) then imply 
\[a(t^px,t^qy)=t^pa(x,y),\ b(t^px,t^qy)=t^qb(x,y).\]
As $t^n\neq 1$ for all $n\geq 1$ we deduce there exist $\alpha,\beta\in\K^*$ such that $a=\alpha x, b=\beta y$: indeed, writing $a=\sum_{ij\geq 0}a_{ij}x^iy^j$ we get that $a_{ij}\neq 0$ implies $pi+qj=0$, hence $a_{ij}\neq 0$ if and only if $(i,j)=(1,0)$; analogously for $b=\sum_{k\ell\geq 0} b_{k\ell}x^ky^\ell$. Then there exists $\gamma\in\K^*$ such that $D=
\gamma(\alpha x\partial_x +\beta y\partial_y).$ Since $x^q-cy^p$ is an eigenvector of $D$ with nonzero eigenvalue we deduce $D=\bar{\gamma}(px\partial_x+qy\partial_y)$ for some $\bar{\gamma}\in\K^*$. But such a derivation admits eigenvectors of the form $x^q-\bar{c}y^p$ for every $\bar{c}\in\K^*$: contradiction.  

In case $b)$ there are coprime integers $p,q\geq 1$ and $c,t\in\K^*$ such that  $h=x^qy^p-c$ and $\rho^N=(t^px,t^{-q}y)$, or even $\rho^N=(y,x)$ if $p=q=1$. By reasoning as above we get  $D=pxa_1\partial_x+qyb_1\partial_y$, for $a_1,b_1\in\pol$. Hence by using once (\ref{eq_isotropy})  we deduce $a_1(x,y)=a_1(t^px,t^{-q}y), b_1(x,y)=b_1(t^px,t^{-q}y)$ and then 
$a_1$ and $b_1$ may be written in the form
\[a_1=\sum_{i}a_{ii}x^{qi}y^{pi},\ b_1=\sum_{i}b_{ii}x^{qi}y^{pi}.\]
It readily follows that $\aut(D)$ contains $G_{p,-q}$  or $G_{p,-q}\rtimes \Z/2\Z$ as stated, depending on $(p,q)\neq (1,1)$ or $(p,q)= (1,1)$. Moreover, the fact that $[x^qy^p-c]$ belongs to $\eig(D)$ is equivalent to saying $x^qy^p-c$ divides $(qa_1+pb_1)x^qy^p$; hence $x^qy^p$ is eigenvector of $D$. If $c\neq 0$, by parametrizing the curve $(x^qy^p-c=0)$ by $u\mapsto (\epsilon u^p,u^{-q})$, where $\epsilon$ is a primitive $q^{\text{th}}$ root of $c$, we deduce that condition signifies there is $\ell\geq 0$ such that 
\[
\sum_{i=0}^\ell (qa_{ii}+pb_{ii})c^i=0, \ qa_{\ell\ell}+pb_{\ell\ell}\neq 0.
\]
Note that if $\ell=0$, i.e. $qa_{ii}+pb_{ii}=0$ for all $i>0$, then the equality above does not depend on $c$ and then $[x^qy^p-e]\in\eig(D)$ for every $e\in\K^*$, which is not possible. Hence $\ell>0$, $c$ is algebraic over $\K$, and \emph{a posteriori}, $\eig(D)$ contains all elements of the form $[x^qy^p-e]$ where $z=e$ is a solution of the equation 
\[\sum_{i=0}^\ell (qa_{ii}+pb_{ii})z^i=0.\]

Finally, we notice that $x$ and $y$ are two rectifiable eigenvectors of $D=pxa_1\partial_x+qyb_1\partial_y$,
which completes the proof.    
\end{proof}

The following result complements Theorem \ref{thm_finite} and may be thought of as a partial converse of it: 

\begin{pro}\label{pro_finite}
Let $p,q,\ell\geq 1$ be integer numbers, with $p,q$ coprime, and 
consider the derivation $D=pxa_1\partial_x+qyb_1\partial_y$ of $\K[x,y]$, where \[a_1=\sum_{i=0}^\ell a_{ii}x^{qi}y^{ip},\ b_1=\sum_{i=0}^\ell b_{ii}x^{qi}y^{ip},\]
such that $qa_{\ell\ell}+pb_{\ell\ell}\neq 0$. If $\{a_{00},b_{00}\}$ is linearly independent over the field $\Q$ of rational numbers, then $\eig(D)$ is finite. 
\end{pro}  
\begin{proof}
  Assume by contradiction that $\eig(D)$ is infinite. Then $D$ admits a rational first integral $f\in\K(x,y)\setminus\K$, that is $D(f)=0$.

  Write $f=g/h$ where $g,h\in\K[x,y]$ have no non-constant common factors; we put $g=\sum_{i\geq n} g_i, h=\sum_{j\geq m} h_j$, where $g_i,h_j$ are homogeneous of degrees $i$, $j$, respectively, and $g_nh_m\neq 0$, for non-negative integers $n,m$. We may assume $n-m\neq 0$, i.e. $c:=g_n/h_m\not\in\K$, since otherwise we may replace $f$ with $f-c$.

  On the other hand, $D(f)=0$ signifies $D(g)h=gD(h)$ which by equaling homogeneous terms of minimal degree implies
  \begin{equation}\label{eq_finite}
  (pxa_{00}\partial_x(g_n)+qyb_{00}\partial_y(g_n))h_m=g_n(pxa_{00}\partial_x(h_m)+qyb_{00}\partial_y(h_m)).
\end{equation}
Now consider the derivation $D_1:=\alpha x\partial_x+\beta y\partial_y$, where  $\alpha=pa_{00}, \beta=qb_{00}$. Then (\ref{eq_finite}) signifies $D_1(g_n/h_m)=0$, and hence $g_n/h_m$ is a rational first integral of $D_1$. In order to complete the proof it suffices to show that  $\{\alpha, \beta\}$ is linearly dependent over $\Q$: contradiction.

In fact, we may assume $g_n,h_m$ have no non-constant common factors. Note that (\ref{eq_finite}) may be rewritten as 
  \begin{equation}\label{eq_finite.1}
   \alpha x(\partial_x(g_n)h_m-g_n\partial_x(h_m))   = \beta y (\partial_y(g_n)h_m-g_n\partial_y(h_m)).
 \end{equation}
 The assumption $n-m\neq 0$ excludes the possibility of having null terms between parentheses in (\ref{eq_finite.1}). Then $g_n$ and $h_m$ admit monomial terms of the form $\gamma x^ry^{n-r}$ and $\delta x^sy^{m-s}$, respectively, for $1\leq r< n$, $1\leq s <m$ and $\gamma\delta\neq 0$. Take $r$ and $s$ to be the biggest ones. By comparing the coefficients of $x^{r+s}y^{n-r+m-s}$ in both sides of (\ref{eq_finite.1}) we deduce
 \[\alpha(r-s)=\beta(n-m-(r-s)),\]
 from which the assertion follows.
\end{proof}

\begin{rems}\label{rem_integrals}

$a)$ When $\K=\C$ Proposition \ref{pro_finite} is a particular case of a general result which says that if a holomorphic vector field admits a rational first integral, then the eigenvalues of its linear part around a singularity are linearly dependent over the field of rational numbers (see \cite{Shi}). In \cite[Thm. 10.1.2]{Now1}, Nowicki treated the case of linear derivations over an arbitrary field of characteristic 0. The last result can be used to generalize the former one to the case of a $\K$-derivation of the ring $\K[[x_1,\ldots,x_n]]$ of formal power series: indeed a polynomial derivation is analytically conjugate to a linear one (\cite{Ma} for the case $\K=\C$).      

$b)$ By part $ii)$ of the statement $b)$ in Theorem \ref{thm_finite} a derivation as in Proposition \ref{pro_finite} always admits an irreducible eigenvector which is not equivalent to an element in $\K[x]$.
\end{rems}  

\begin{exa}
  Consider the polynomial $f=y^2+x^n+1$, where $n\geq 2$. Note that $f$ is irreducible and $f=0$ is a smooth curve whose genus is positive if $n$ is odd and $\geq 3$ or even and $n\geq 6$ (see \cite[Chap. III, \S 6.5]{Sh}). A derivation $D=a\partial_x+b\partial_y$ admits $f$ as an eigenvector if and only if there is $\lambda\in\K[x,y]$ such that $nax^{n-1}+2by=\lambda f$. We look for such a $D$ with $a=xa_1+Q(y), b=yb_1+P(x)$, where $P\in\K[x], Q\in\K[y]$, which requires $\lambda=nx^{n-1}Q+2yP=na_1=2b_1$. If in addition we assume $P(x)=\alpha x, Q(y)=\beta y$ for suitable $\alpha,\beta\in\K^*$, then we get a derivation as before whose linear part is $\beta y\partial_x+\alpha x\partial_y$.

  Now, as in the case of Proposition \ref{pro_finite} we deduce $D$ has not a rational first integral for general $\alpha, \beta$. Thus we have constructed a derivation where $\eig(D)$ is finite and contains $[f]$, and we know that if $n\gg 0$, then $f$ is not equivalent to an element in $\K[x]$.

  Moreover, if $n>6$, then one may prove $f$ has genus $\geq 2$, hence $\aut(f)$ is finite. By arguing as in the beginning of the proof of Theorem \ref{thm_finite} we infer $\aut(D)$ is finite.

  Finally, note that if $n$ is even $\aut(D)$ contains the involution $\rho=(-x,-y)$ which shows that $\aut(D)$ may be nontrivial in Theorem \ref{thm_finite}.
\end{exa}

\subsection{Derivations whose kernel contains rectifiable elements}\label{loc_nil}

Let $D$ be a derivation of $\pol$. Assume $\ker D$ contains a rectifiable element, or equivalently that there is $\varphi\in\autpol$ such that the kernel of $\varphi^{-1} D\varphi$ is different from $\K$ and contained in $\K[x]$: indeed, the last condition implies $\varphi^{-1} D\varphi=b\partial_y$ for some $b\in\pol$ whose kernel contains $x$. The simplest examples of that type of derivation belong to the family of locally nilpotent derivations. Recall that  $D$ is locally nilpotent if for every $f\in\pol$ there is $n\geq 0$ such that $D^n(f)=0$. By a result of Rentschler (\cite{Re}) we know that a locally nilpotent derivation in $\pol$ is conjugate to a derivation of the form $u(x)\partial_y$, where $u\in\K[x]$, and hence $\ker D$ contains rectifiable elements. 

\begin{pro}\label{pro_locnil}
	Let $D\in\derpol$ be a derivation and let $\varphi\in\autpol$ be an automophism  such that $\varphi^{-1}D\varphi=u(x)\partial_y$. Then	

$a)$ $h:=\varphi(x)$ is a rectifiable element such that $\ker D=\K[h]$.
	
$b)$ The conjugate $\varphi^{-1}\aut(D)\varphi$ of $\aut(D)$ is
	\[ \{(\alpha x+\beta, \gamma y+s(x)) | \alpha, \gamma \in \K^*, \beta \in \K, s(x) \in \K[x], u(\alpha x+\beta)/u(x)=\gamma \}\]
In particular, $\aut(D)$ is not an algebraic group.  
\end{pro}
\begin{proof}
Assertion a) is straightforward and to prove assertion b) we only need to compute the isotropy of $u\partial_y$. In fact, $\rho=(f,g)\in\aut(u\partial_y)$ if and only if $uf_y=0, ug_y=u(f)$. Then $f,g_y\in\K[x]$, $u$ divides $u(f)$ and $g_y=u(f)/u$. Moreover, since $f_y=0$, the Jacobian condition says $f_xg_y\in\K^*$, from which we deduce $f=\alpha x+\beta$ and $g=\gamma y+P$, where $\alpha, \beta,\gamma\in\K$, with $\alpha\gamma\neq 0$, $P\in\K[x]$ and $u\gamma=u(f)$. The proof is then complete.
\end{proof}

\begin{cor}\label{cor_locnil} The isotropy group of a locally nilpotent derivation $D \in \derpol$ is conjugate to a group of automorphisms of the affine plane $\K^2$ which preserve a curve of equation $(u(x)=0)$, with $u\in \K[x].$ More precisely, if $\varphi^{-1}D\varphi=u(x)\partial_y$, then $\varphi^{-1}\aut(D)\varphi\subset\aut(u)$.
	
\end{cor}
\qed

\begin{exa}\label{exa_locnil}
  Keeping notations as in Proposition \ref{pro_locnil}, we have a homomorphism $\aut(D)\to(\K^*\rtimes\K)\times\K$, $(\alpha x+\beta, \gamma y+s(x))\mapsto (\alpha\rtimes \beta, \gamma)$, whose image subgroup is defined by the equation $u(\alpha x+\beta)=\gamma u(x)$ and its kernel subgroup is $\{(x,y+s(x)); s\in\K[x]\}$; clearly $\aut(D)$ is an extension of the first subgroup by the last one.

  Note that the equation above we have referred to may be interpreted as saying the map $\K\to\K$, giving by $t\mapsto \alpha t+\beta$, permutes the roots of $u(x)=0$ and $\alpha^n=\gamma$. Hence there are not many possibilities for $\alpha, \beta$ and $\gamma$ unless $\beta=0$ and $u(x)=x^n$ is a power of $x$. In that case, for any $\gamma\in\K^*$ we obtain $\alpha$ is any of the $n$-roots of $\gamma$.
\end{exa}

\noindent{\bf Proof of Theorem \ref{thm_B}}. The ``only if'' part of the first assertion is straightforward. To prove the corresponding converse part suppose there is a rectifiable element $h\in\ker D$; let $\phi\in\autpol$ such that $\phi(h)=x$. Hence $x\in \ker(\phi D\phi^{-1})$, from which it follows $\phi D\phi^{-1}=b\partial_y$ for some $b\in\pol\setminus\{0\}$, as required.  
	
	Assertion a) follows readily form Proposition \ref{pro_locnil}. 
	
	To prove b) we may suppose $D=b\partial_y$. Write $b=b_0(x)+\cdots+b_n(x)y^n$ with $b_0,\ldots,b_n\in\K[x]$ and assume $n\geq 1$. Take $\rho\in\aut(D)$ and put $f=\rho(x), g=\rho(y)$. Since $x\in\ker D$ and $\partial(f,g)/\partial(x,y)\in\K^*$ the first equality in (\ref{eq_isotropy}) gives $f=\alpha+\beta x$, $g=\delta y+P(x)$ for some $\alpha,\beta,\delta\in \K$, with $\beta\delta\neq 0$, and $P\in K[x]$. The second equalty in (\ref{eq_isotropy}) gives 
	\[\sum_{i=0}^n \{b_i^{\alpha\beta}(\delta y+P(x))^i-\delta b_i(x)\}y^i=0,\]
	where we have denoted $b_i^{\alpha\beta}:=\rho(b_i)$ for $i=0,\ldots, n$. We deduce
        \[\sum_{i=0}^n b_{i}^{\alpha\beta}P(x)^i=\delta b_0\]
from which it follows that either $P=0$ or $\deg P$ is bounded. Hence Corollary \ref{cor_ind3} implies $\aut(D)$ is an algebraic group which proves $b)$ and completes the proof of the theorem.

\qed

\begin{exa}
	Consider the derivation $D=b\partial_y$ with all notations as in the proposition above and $b=b_ny^n$, $b_n\in\K^*$  and $n\geq 1$. Then $\delta^{n-1}=1$ and $P=0$. We deduce 
	\[\aut(D)\simeq (\K^*\rtimes\K)\times \mu_{n-1},\]
	where $\mu_{n-1}$ is the cyclic group of order $n-1$, and $\K^*\rtimes\K$ and $\mu_{n-1}$ act on $\autpol$ by $(\beta,\alpha)\to (\alpha+\beta x,y)$ and $\delta\mapsto (x,\delta y)$, respectively.   
\end{exa}

\begin{cor}
	$b\partial_y$ is locally nilpotent if and only if $\aut(D)$ is not an algebraic group.  
\end{cor}

We finish this section with the following:

\noindent{\bf Question} Does the ``non algebricity''  of $\aut(D)$ characterize locally nilpotence for derivations of $\pol$?

\subsection{Shamsuddin derivations}\label{sham}
A derivation $D$ of $\K[x,y]$ is said to be a \emph{Shamsuddin derivation} if $D$ is of the form $$D=\partial_x+(a(x)y+b(x))\partial_y,$$ where $a(x),b(x) \in \K[x]$. In \cite{Sha} there is a criterion which allows to decide whether D is simple or not, depending on whether the equation $D(h) = ah+b$ has no solution $h \in \K[x]$ or admits such a solution, respectively. It follows that a Shamsuddin
derivation with $a=0$ is not simple.  Furthermore, we know that  
if $D$ is a Shamsuddin derivation with $a\neq 0$ then $D$ is simple if and only if $\aut(D) = \{id\}$ (see \cite[Thm. 6]{Baltazar} and \cite[Thm. 1]{MePa}).

In the next proposition we describe the isotropy of Shamsuddin derivations and in particular, we complement the results given in the second reference above relatively to that type of derivations (for a polynomial $ P\in\K[x]$ we denote $P'(x):=\partial_x P(x)$). Let us first state a preliminary result whose proof is straightforward.

\begin{lem}\label{lem_shum}
 Let $\cale_\lambda(a,b)$ be the differential equation $D(h)=ah+\lambda b$, where $\lambda\in\K$ and $D$ is a Shamsuddin derivation as above. Then

  a) $\cale_\lambda(a,b)$ has at most one solution in $\K[x]$ which exists and is null if and only if $\lambda b=0$.

  b) If $P_1,P_2\in\K[x]$ are solutions of $\cale_{\lambda_1}(a,b),\cale_{\lambda_2}(a,b)$, respectively, then $-\lambda_2P_1+\lambda_1P_2=0$.

  c) If $b\neq 0$ and $\cale_\lambda(a,b)$ admits a solution in $\K[x]$ for some $\lambda\neq 0$, then  $\cale_\mu(a,b)$ admits a solution in $\K[x]$ for any $\mu\in\K$. 
\end{lem}

Note that if $b\neq 0$ and $\deg_x b<\deg_x a$, then $\cale_\lambda(a,b)$ admits no solution in $\K[x]$. 

\begin{pro}
  Let $D=\partial_x+(a(x)y+b(x))\partial_y$ be a  Shamsuddin derivation. We have the following assertions:

  a) If $a=0$, then $D$ is conjugate to $\partial_x$.

  b) If $a\neq 0$, then exactly one of the following holds:
  \begin{itemize}
  \item[i)] $(a,b)\in\K^*\times\K$ and $\aut(D)$ is isomorphic to  $\K\times\K^*$ or  $\K\times(\K\ltimes\K^*)$ depending on whether $b=0$ or $b\neq 0$, respectively.
    \item[ii)]  $\deg a \geq 1$ or $\deg b \geq 1$ and $\aut(D)$ is isomorphic to either $\K^*$, if $b=0$ or $b\neq 0$ and $\cale_\lambda(a,b)$ admits a solution in $\K[x]$ for some $\lambda\neq 0$, or trivial otherwise in which case $D$ is simple. 
\end{itemize}
In particular, $\aut(D)$ is algebraic if and only if $a\neq 0$. 
\end{pro}
\begin{proof}
  If $a=0$ we take $B\in\K[x]$ such that $B'(x)=b(x)$ and consider the automorphism $\varphi=(x,y+B(x))$. We have $\varphi D\varphi^{-1}=\partial_x$ which proves a).

  On the other hand, if  $a\neq 0$ then \cite[Prop. 9]{MePa} says either i) holds or at least one of the polynimials $a,b$ has degree 1 and for any $\varphi\in\aut(D)$ there exist $d\in\K^*$ and $P\in\K[x]$ such that $\varphi=(x,dy+P(x))$ and $P'-aP=(1-d)b$. Note that for such a $\varphi$ the polynomial $P$ belongs to $\cale_{1-d}(a,b)$.

  Since $\cale_\lambda(a,0)$ consists of the null solution we easily deduce $b=0$ implies $\aut(D)\simeq \K^*$.

  Now assume $b\neq 0$ and $\aut(D)\setminus \{id\}\neq \emptyset$. Hence Lemma \ref{lem_shum} implies
  \[\aut(D)=\{(x,dy+(1-d)P); d\neq 0\}\simeq\K^*\]
 
  Finally, $b\neq 0$ and $\aut(D)=\{id\}$ if and only if $\cale_\lambda(a,b)$ admits no solution for any $\lambda\neq 0$ which corresponds to saying $D$ is simple as we have said before.
\end{proof}

To finish this subsection we describe $\eig(D)$ for a Shumsuddin derivation $D$; note that the case $a=0$ is completely understood, then we disconsider that case. The following result is a straightforward consequence of \cite[Thm. 4.1b)]{BLL}: 

\begin{pro}
Let $D=\partial_x+(a(x)y+b(x))\partial_y$ be a Shumsuddin derivation with $a\neq 0$. Then $D$ is simple if and only if $\eig(D)\neq\emptyset$. Moreover, in this case $\eig(D)=\{[y-P(x)]\}$ where $P\in\K[x]$ is the unique solution of $\cale_1(a,b)$. 
\end{pro}

\section{Derivations with nontrivial kernel}\label{sec_nontrivial}

Before treating the subject of this section we characterize when the automorphism group of a non-constant polynomial map $h:\K^2\to\K$ is an algebraic group, which is interesting in its own. The motivation here is that the isotropy of a derivation whose kernel is $\K[h]$ may be naturally embedded into that automorphism group.

\subsection{Automorphism of a polynomial morphism}\label{subsec_autpol}

For a non-constant polynomial $h\in\pol$ we define $\autpolh$ to be the subgroup of $\autpol$ whose elements are the automorphisms $\varphi$ such that there are $\alpha,\beta\in\K$, with $\alpha\neq 0$, such that $\varphi(h)=\alpha h+\beta$. In the case where $h$ is a reduced polynomial that group contains the group of plane automorphisms which stabilize the curve $h=0$.

If we thought of $\varphi\in\autpol$ and $h$ as morphisms of algebraic varieties $\varphi:\K^2\to\K^2$ and $h:\K^2\to\K$, respectively, then $\varphi\in \autpolh$ if and only if there is an unique automorphism $\tau=\tau_\varphi:\K\to\K$ such that the following diagram commutes
\[\xymatrix{\K^2\ar@{->}[r]^\varphi\ar@{->}[d]_h&\K^2\ar@{->}[d]^h\\
  \K\ar@{->}[r]^\tau&\K}\]

In this subsection we prove that $\autpolh$ is an algebraic group if and only if $h$ is not equivalent to an element in $\K[x]$ which generalizes \cite[Thm.1]{BlSt} in the case of algebraically closed fields.

First of all note that $\autpolh$ admits a structure of ind-group: indeed, the projection $\autpol\times (\K^*\rtimes\K)\to \autpol$ respects the ind-structure associated to the ind-group $\autpol=\lim_{d}\aut(\pol)_d$ and for any $d\geq 1$ establishes an isomorphism between the closed subset $\{(\rho,(\alpha,\beta)); \rho(h)=\alpha h+\beta\}$ of $\aut(\pol)_d\times (\K^*\rtimes\K)$ and $\autpolh\cap\aut(\pol)_d$.

On the other hand, there is a natural homomorphism $\kappa:\autpolh\to\K^*\rtimes\K$, $\varphi\mapsto \tau_\varphi$, whose kernel is $\fix(h)$ and whose image will be denoted by $G_h$. Note also that we have $\aut(h)=\kappa^{-1}(\K^*\rtimes 0)$. We assert that $\kappa$ is a morphism of ind-groups: in fact, if $\varphi=(f,g)\in\aut(\pol)_d$, for some $d\geq 1$, and $h=\sum a_{ij}x^iy^j$, then $\kappa(\varphi)=(\alpha,\beta)$ if and only if $h(f,g)=\alpha h+\beta$. Note that if $a_{ij}\neq 0$, then $a_{ij} f^ig^j$ may be written in the form $\sum_{k\ell} A^{ij}_{k\ell}x^ky^\ell$, where $A_{k\ell}$ is a polynomial expression over  $\Q$ in the coefficients of $f$ and $g$ relative to the standard basis (in some order) of the vector space of polynomials of degree $\leq d$.

Since $h\not\in\K$ we know there exist $r,s$ with $r+s\geq 1$ such that $a_{rs}\neq 0$. We deduce\[a_{rs}\alpha=\sum_{ij} A^{ij}_{rs},\ \beta=-a_{00}\alpha+\sum_{ij} A^{ij}_{00}.\] 
	Hence $\alpha$ and $\beta$ are regular functions of $\aut(\pol)_d$ for all $d\geq 1$, which proves the assertion.

There is an exact sequence of groups 
\begin{equation}\label{eq_algroup}
  \xymatrix{1\ar@{->}[r]&\fix(h)\ar@{->}[r]&\autpolh\ar@{->}[r]^{\ \ \ \ \kappa}&G_h\ar@{->}[r]&1}
\end{equation}
If $\autpolh$ is algebraic, and since the closure $\overline{G_h}$ of $G_h$ in $\K^*\rtimes\K$ does, then $G_h=\overline{G_h}$ is algebraic too. In that case $G_h$ admits an irreducible decomposition $G^0\cup G^1\cup\cdots\cup G^\ell$, where $G^0$ is the unique component containing the identity $1\in G_h$. Then $G^0$ is a normal subgroup of $G_h$ and every component $G^i$ is a suitable coset of $G^0$; we set $\autpolh^i:=\kappa^{-1}(G^i)$ for $i=0,\ldots,\ell$. It follows $\autpolh^0$ is normal and closed in $\autpolh$. Since all $\autpolh^i$'s are cosets of $\autpolh^0$ we deduce $\autpolh^0$ is the irreducible component of $\autpolh$ containing $1$.

Finally, if $\P^2$ is the natural completion (see \S \ref{subsec_pencil}) of $\K^2$ defined by adding a projective line $L_\infty$ at infinity, $\P^2=\K^2\cup L_\infty$, the closure of the fibers of $h:\K^2\to\K$ define a pencil $\Lambda$ without fixed part and such that $\emptyset\neq \bas(\Lambda)\subset L_{\infty}$. We may identify $\autpolh$ with $\bir((\P^2,L_\infty),\Lambda)$.

We consider $\K^2\subset \P^2$ by means of the map $(a,b)\mapsto (a:b:1)$. If $x,y$ are coordinates on $\K^2$ and $u,v,w$ are homogeneous coordinates on $\P^2$, then $x=u/w, y=v/w$ and $L_\infty=(w=0)$. In that case $\Lambda$ is composed by lines through the point $p=(0:1:0)$ if and only if $h\in\K[x]$.   

\begin{thm}\label{thm_autpolh}
 Let $h\in\pol$ be a non-constant polynomial. Then $\autpolh$ is an algebraic group if and only if $h$ is not equivalent to an element in $\K[x]$. Moreover, we have one of the following three possibilities:

  $a)$ $\aut(h)$ is finite and one of the following assertions holds:
  \begin{itemize}
  \item[$a_1$)] $\autpolh=\aut(h)$.
    \item[$a_2$)] $\autpolh^0=\K$ and all fibers of the map $h:\K^2\to\K$ are smooth and isomorphic.
    \end{itemize}

 $b)$ $\aut(h)$ is infinite, coincides with $\autpolh$ and up to conjugate by an element in $\autpol$ there are coprime integers $p,q\geq 1$ and $c\in\K^*$ such that one of the following possibilities occurs
\begin{itemize}
\item[$b_1)$] $p,q>1$, $h=(x^q-cy^p)^n$ for some $n\geq 1$, $\autpolh=G_{p,q}$.
\item[$b_2)$] $h=(x^qy^p-c)^n$ for some $n\geq 2$ and $\autpolh=G_{p,-q}$ or $G_{p,-q}\rtimes \Z/2\Z$ depending of $(p,q)\neq (1,1)$ or $(p,q)=(1,1)$, respectively.
  \item[$b_3)$] $h=x^qy^p-c$ and $\autpolh/\aut(h)=\K^*\rtimes\K$.
\end{itemize}
$c)$  $\aut(h)$ is infinite, coincides with $\autpolh$ and $h$ is a product of at least two rectifiable elements.
  \end{thm}

  \begin{proof}
    We identify $\autpolh$ with $\bir((\P^2,L_\infty),\Lambda)$. By taking into account the commentaries above the first assertion is consequence of Corollary \ref{cor_blst4} and \cite[Lem. 2.6]{BlSt}. To prove the rest of the theorem we assume $h$ is not equivalent to an element in $\K[x]$, i.e. $\autpolh$ is an algebraic group. Recall $h_0$ denotes the product of all irreducible components of $h$.

    According to \cite[Thm. 1 and Lem. 4.1]{BlSt} $\aut(h_0)$ is an affine algebraic group. We deduce there exists a positive integer $m$ such that every $m^{th}$ power of an element in $\aut(h_0)$ stabilizes all irreducible components of $h_0$.

    First suppose $\aut(h_0)$ is finite. Consider the exact sequence of algebraic groups
    \[\xymatrix{1\ar@{->}[r]&\fix(h)\ar@{->}[r]&\autpolh^0\ar@{->}[r]^{\ \ \ \ \ \ \kappa_0}&G^0},\  \kappa_0:=\kappa|_{\autpolh^0}\]
    where $G^0$ is the irreducible component of $G_h=\overline{G_h}$ containing 1. By connectedness we get $\fix(h)=\{1\}$ and then $\autpolh^0=G^0$. Since $\aut(h)=\kappa^{-1}(\K^*\rtimes 0)$ is finite we get $G^0=1$ or $G^0=\K$ (see Lemma \ref{lem_exe}). In the first case $\autpolh=\aut(h)$.

    On the other hand, if $G^0=\K$, then $\autpolh$ acts transitively on the set of fibers of the map $h:\K^2\to\K$. Since a general fiber of $h$ is smooth the assertion a) follows. 
    
  Now we assume $\aut(h_0)$ is infinite. Hence it has positive dimension and Lemma \ref{lem_agroup2} implies that group is not a torsion group, so $\aut(h)$ neither. By Lemma \ref{lem_h} it suffices to consider the cases where $h$ is as in the assertions $b_1)$, $b_2)$ or $b_3)$, or it may be written as $h=x^{n_1}h_2^{n_2}\cdots h_\ell^{n_\ell}$, with $h_j$ rectifiable and not belonging to $\K[x]$ for any $j$.

  In the first case $\aut(h)=\aut(h_0)$ and we have two possibilities: $(i)$ $h=(x^q-cy^p)^n$  and $(ii)$ $h=(x^qy^p-c)^n$, with notations as above.

  Note that in $(i)$ the unique singular fiber of the map $h:\K^2\to\K$ is supported on $x^q-cy^p=0$, hence $\autpolh=\aut(x^q-cy^p)$ and the assertion $b_1)$ is proven. 

  In $(ii)$, if $n>1$ by arguing as above we deduce $\autpolh=\aut(x^qy^p-c)$ from which $b_2)$ follows. Otherwise $h=x^qy^p-c$. Then there is a homomorphism $\eta:\K^*\times\K\to \autpolh$ defined by $(\gamma,\delta)\mapsto (\gamma x,\delta y)$ such that the image of $\kappa\eta$ is $\K^*\rtimes\K$. Since in this case $\fix(h)=\aut(h)$ the assertion $b_3)$ is consequence of (\ref{eq_algroup}). 

Finally, suppose we are in the last case. Hence the curve $h=0$ admits a singularity of the form $(0,y_0)\in\K\times\K$.

Take $\varphi\in \kappa_0^{-1}(1\rtimes \K)$. Then there is $\beta\in\K$ such that $\varphi^r(h)=h+r\beta$ for $r=1,2,\ldots$ Since $h:\K^2\to\K$ has smooth general fibers we deduce the subset  $\{\beta, 2\beta, 3\beta,\ldots\}$ of $\K$ is finite, and then it contains $0$. Hence $\beta=0$. In other words, $\kappa_0^{-1}(1\rtimes \K)\subset \kappa_0^{-1}(\K^*\rtimes 0)$ from which it follows $\autpolh^0=\kappa_0^{-1}(\K^*\rtimes 0)$. Since this group is the connected component of $\aut(h)$ which contains 1 we deduce $\autpolh=\aut(h)$, and this completes the proof of the theorem.
\end{proof}

 \begin{rem}\label{rem_artal}
  Polynomials as in the part $a_2)$ of Theorem \ref{thm_autpolh} were constructed in \cite{ACNL}. We do not know how to calculate their corresponding automorphism group.
 \end{rem}
\subsection{Derivations with nontrivial kernel}\label{subsec_nontrivial}

Let $D$ be a derivation of $\K[x,y]$ such that $\ker D=\K[h]$ for a polynomial $h\in\pol\setminus\K$. Note that necessarily the fibers of the map $h:\K^2\to\K$ are connected because $\K[h]$ is integrally closed in $\pol$ (\cite[Prop. 2.2]{NaNo}). In particular, there is no $n>1$ and $f\in\pol$ such that $h=f^n$.

The aim here consists of two things. First of all we determine when $\aut(D)$ is an algebraic group and then, for such a derivation, we classify the isotropy group of $D$ in terms of the generator $h$ of $\ker D$ as $\K$-algebra. We already treated the case where $\ker D$ contains rectifiable elements (see  Theorem \ref{thm_B}), hence from now on we assume there is no $\varphi\in\autpol$ such that $\varphi(h)\in\K[x]$.

In order to fix the first objective it suffices to embed $\aut(D)$ in $\autpolh$ as a closed ind-subgroup, by Theorem \ref{thm_autpolh}. And this is clear because the commutation relationship $\rho D=D\rho$: indeed, it defines a closed subset in $\autpolh\cap\aut(\pol)_d$ for any $d\geq 1$ (see the proof of Corollary \ref{cor_ind3}), and it also implies $\rho$ induces a $\K$-automorphism of $\K[x]$ from which one deduces $\rho\in\autpolh$.  

Before stating the main result of \S \ref{sec_nontrivial} we introduce the last terminology and give a proposition which has its own interest. As we have said in \S \ref{sec_intro}, a derivation $D=a\partial_x+b\partial_y$ is said to be \emph{irreducible} if $a,b$ have no non-constant common factors. If $D$ is arbitrary we may write $D=gD_1$ where $g\in\pol$ and $D_1$ is irreducible. 

\begin{pro}\label{pro_main}
  Let $D=gD_1\in\derpol$ be a nonzero derivation with $D_1$ irreducible and $g\in\pol$ non-constant. Then we have

  $a)$ $\aut(D)$ is an ind-subgroup of $\aut(g)$ contained in $\autsc(D_1)$.

  $b)$  $\aut(D)$ is isomorphic as ind-group to a closed subgroup of $\aut(g)\times \autsc(D_1)$.
  \end{pro}
  \begin{proof}
    Let $\rho\in\aut(D)$ and write $D_1=a_1\partial_x+b_1\partial_y$. The result is trivial when $a_1b_1=0$, hence assume $a_1b_1\neq 0$.

    Let us take an arbitrary linear polynomial $l=\alpha x+\beta y$ and set $f:=\rho(l)$. Note that $D\neq 0$ implies $D(l)\neq 0$ for general $\alpha$ and $\beta$.

    We have
    \[\rho(g)(\rho(a_1)\alpha+\rho(b_1)\beta)=g(a_1f_x+b_1f_y).\]
    If $a_1,b_1\in\K$, then $g$ divides $\rho(g)$. Otherwise, since $g$ admits a finite number of divisors we deduce $\rho(a_1)\alpha+\rho(b_1)\beta$ divides $a_1f_x+b_1f_y$ for $(\alpha,\beta)$ in a Zariski dense open subset of $\K^2$, say $U\subset \K^2$. Hence $g$ divides $\rho(g)$ in any case.

    By interchanging the roles of $f$ and $l$ we also get $\rho(g)$ divides $g$ from which it follows $\aut(D)$ is contained in $\aut(g)$. Since $\aut(D)$ is closed in $\autpol$ we deduce it is closed in $\aut(g)$.

    On the other hand, there exists $\lambda\in\K^*$ such that $\rho(a_1)\alpha+\rho(b_1)\beta=\lambda(a_1f_x+b_1f_y)$ for any $(\alpha,\beta)\in U$, that is $\rho D_1-\lambda D_1\rho$ vanishes on a general polynomial of degree 1, hence it does on every such polynomial. We deduce $\rho D_1=\lambda^{-1} D_1\rho$, and then $\rho\in\aut(g)\cap \autsc(D_1)$ which completes the proof of the assertion $a$).
    
Finally, we consider the map $\Theta:\aut(g)\times \autsc(D_1)\to \K^*\times\K^*$ defined by setting $\Theta(\sigma,\rho)=(\gamma,\lambda)$ if and only if $\sigma(g)=\gamma g, \rho D_1=\lambda^{-1} D_1\rho$. We conclude that the image of the natural immersion $\aut(D)\to \aut(g)\times \autsc(D_1)$ is precisely the inverse image of the diagonal in $\K^*\times\K^*$ under $\Theta$, hence $b)$ follows.
\end{proof}

We consider the restriction $\kappa_D$ of the map $\kappa:\autpolh\to\K^*\rtimes\K$ to $\aut(D)$. If $\aut(D)$ is algebraic, then its image under $\kappa_D$ does and we denote by $\aut(D)^0$ the inverse image of the irreducible component of that image group containing the unity.

Now, Proposition \ref{pro_main} may be used together with the following:

\begin{thm}\label{thm_ker}
  Let $D$ be a derivation of $\pol$ such that $\ker D=\K[h]$ for some $h\in\pol\setminus\K$. If $h$ is not equivalent to an element in $\K[x]$, then $\aut(D)$ and $\autsc(D)$ are algebraic groups. Moreover, if $D=gD_1$ with $D_1$ irreducible, then we have the following assertions:

  $a)$ $\aut(h)$ is finite and one of the following holds
  \begin{itemize}
  \item[$a_1$)] $\aut(D_1)\subset\aut(h)$.
    \item[$a_2$)] $\aut(D_1)^0=\K$ and all fibers of $h:\K^2\to\K$ are smooth, irreducible and isomorphic. 
    \end{itemize}
 $b)$ $\aut(h)$ is infinite and up to conjugate by an element in $\autpol$ there are coprime integers $p,q\geq 1$ and $c\in\K^*$ such that one of the following possibilities occurs
\begin{itemize}
\item[$b_1)$] $p,q>1$, $h=x^q-cy^p$, $D_1=cpy^{p-1}\partial_x-qx^{q-1}\partial_y$ and  $\aut(D_1)=G_{p,q}$.
\item[$b_2)$] $h=x^qy^p-c$, $D_1=px\partial_x-qy\partial_y$ and $\aut(D_1)=G_{p,-q}$ or $G_{p,-q}\rtimes \Z/2\Z$ depending on $(p,q)\neq (1,1)$ or $(p,q)=(1,1)$, respectively.
\end{itemize}
$c)$  $\aut(h)$ is infinite, $\aut(D_1)\subset\aut(h)$ and up to conjugate by an element in $\autpol$ we have $h=x^{n_1}h_2^{n_2}\cdots h_\ell^{n_\ell}$ with $h_i$ rectifiable and not in $\K[x]$, $i=2,\ldots,\ell$.
  \end{thm}
  \begin{proof}

   By what we have explained in the third paragraph of this subsection we already know that $\aut(D)$ is an algebraic group. Analogously, we may embed $\autsc(D)$ as a closed subgroup of $\autpolh\times\K^*$ which implies it is an algebraic group too. Assertion $a)$ follows readily from Theorem \ref{thm_autpolh}. Moreover, we also deduce that if $\aut(h)$ is infinite, then $h$ takes one of the forms $(x^q-cy^p)^n$, $(x^qy^p-c)^n$ or $h$ is as in $c)$; then $c)$ is already proven. In the first case $D(x^q-cy^p)=0$ and we deduce $n=1$; analogously for the second case.

    Now, if $D=a\partial_x+b\partial_y$ and $h=x^q-cy^p$, then $qx^{q-1}a+pcy^{p-1}b=0$, hence $D_1=cpy^{p-1}\partial_x-qx^{q-1}\partial_y$. Analogously, if  $h=x^qy^p-c$, then $qay+pby=0$ and so $D=px\partial_x-qy\partial_y$ for a polynomial $g_1\in\pol$. A straightforward calculation gives $\aut(D_1)$ is as stated in each case which proves $b)$ and completes the proof.
\end{proof}

\begin{exa}
If $r,s$ are positive integers, then a derivation of the form $D=g(sy^{s-1}\partial_x-rx^{r-1}\partial_y)$, with $g\in\pol$ different from 0, verifies $\ker D=\K[x^ry^s]$. As we have seen $\aut(x^ry^s)\simeq\K^*\times\K^*$ (see Example \ref{exas_rs}a)). A straightforward calculation gives $\aut(D)$ is a cyclic group of order $rs-r-s$ when $g\in\K^*$, hence $\aut(D)$ is finite for any $g$. An analogous conclusion follows for a derivation such that $\ker D=\K[x^ry^s(y-x^n)]$, for a positive integer $n$ (cf. Example \ref{exas_rs}b)).
\end{exa}

The example above motivates the following:

\begin{con}
  In case $c)$ of Theorem \ref{thm_ker} we have $\aut(D)$ is finite.
\end{con}

\section{derivations with $\eig(D)=\infty$ and $\ker D=\K$}

We assume $D$ to be an irreducible derivation with $\eig(D)=\infty$ and $\ker D=\K$; as we have pointed out in Remark \ref{rem_integral} this situation actually occurs. Hence $D=a\partial_x+b\partial_y$, with $a,b$ polynomials having no non-constant factors, $ab\neq 0$ and $D$ admits a rational first integral $f/g$ where $f,g\in\K[x,y]\setminus\K$ have no nontrivial common factors.

\begin{exa}
    If $f,g\in\K[x,y]$, then the derivation $D_{f,g}=(f_yg-fg_y)\partial_x-(f_xf-fg_x)\partial_y$ admits $f/g$ as rational first integral. Moreover, a straightforward calculation shows $D(f/g)=0$ if and only if there are $a,b\in\K[x,y]$ such that $aD=bD_{f,g}$.
    \end{exa}

Except for the second part of the proof of Theorem D below, what we will develop in this section is well known the specialists of (for example) holomorphic singular foliations.   

Recall that a rational function $\tau=P/Q\in\K(x)\setminus\K$, with $P,Q\in\K[x]$ polynomials without nontrivial common factors, is said to have \emph{degree} $d\geq 1$ if $d:=\max\{\deg P,\deg Q\}$. The degree 1 rational functions define the $\K$-automorphisms of $\K(x)$.

Now, if $\tau\in\K(x)$, then $\tau(f/g)$ is also a rational fist integral for $D$. We will say $f/g$ is \emph{minimal} if there are not another rational first integral $f_1/g_1$ and a rational function $\tau\in\K(x)\setminus\K$ of degree $\geq 2$ such that $f/g=\tau(f_1/g_1)$ (here we were inspired by \cite[\S 3]{NaNo} and will try in Lemma \ref{lem_rat} below to get a result analogous to Proposition 3.3 therein).

Note that $f/g$ may be thought of as a rational map $\K^2\tor\K$. By  considering the natural completion of the (affine) plane $\P^2=\K^2\cup L_infty$ where $L_infty$ is the line at infinite, we then may extend $f/g$ to define a rational map $F:\P^2\tor\P^1=\K\cup\{\infty\}$ which maps a general point in each component of the curve $g=0$ onto $\infty\in\P^1$; note that a general fiber of that map consists of the closure in $\P^2$ of a curve of equation $f+\alpha g=0$ for some $\alpha\in\K$. Moreover, by resolving the indeterminacy of $F$ (by means of the finite number of blow-ups of the so-called base points of $F$, i.e. the proper and infinitely near common zeros of $f$ and $g$ over $\P^2$) we get a morphism $\widehat{F}:X\to\P^1$ whose general fibers are desingularizations of the closure of general fibers of $F$.

Analogously, a rational function $\tau\in\K(x)\setminus\K$ may be thought of as a rational map $\K\tor\K$ and defines a finite morphism $\P^1\to\P^1$ that we still denote by $\tau$. The degree of $\tau$ is precisely the number of points in a general fiber when we are thinking of it as a morphism. Then $\tau$ has degree 1 if and only if it corresponds to an element in the automorphism group $\pgl(2,\K)$ of $\P^1$.    
 
\begin{lem}\label{lem_irr}
If $f/g$ is minimal, then for a general $\alpha$ the polynomial $f+\alpha g$ is irreducible. 
\end{lem}
\begin{proof}
  Assume for a moment that  $f+\alpha g$ is reducible for any $\alpha$. By Bertini's Theorem the corresponding morphism $\widehat{F}:X\to\P^1$ have disconnected fibers. Stein Factorization Theorem then implies there exists a morphism $G:X\to B$ onto a smooth projective curve and a finite morphism $\eta:B\to\P^1$ of degree $d\geq 2$ such that $\eta G=\widehat{F}$. Since $F$ admits at least a base point we know that $X$ contains a rational curve which maps onto $B$, so $B\simeq \P^1$. We conclude $\eta$ induces a rational function $\tau\in\K(x)$ of degree $d>1$ and $G$ induces a rational map $f_1/g_1:\K^2\tor\K$ such that $f/g=\tau(f_1/g_1)$: contradiction. 
\end{proof}

The set of \emph{singular points} of $D$ is by definition the subset $\sing(D)$ of $\K^2$ which consists of all common zeros of $a=b=0$, or equivalently, the maximal ideal associated to the ideal generated by $a$ and $b$. Clearly $D$ irreducible is equivalent to $\sing(D)$ being finite (maybe empty). Note that $\sing(D)$ may be thought of as the set whose elements are the maximal ideals stable under $D$.

Denote by $\ideala=\ideala_{f,g}$ the ideal generated by $f,g$, which has height $2$. Since $D(\ideala)\subset \ideala$ we deduce $D$ stabilizes all maximal ideal associated to $\ideala$ (\cite{Sei}). If $(x_0,y_0)\in\K^2$ is a point such that $f(x_0,y_0)=g(x_0,y_0)=0$, then we deduce $(x_0,y_0)\in\sing(D)$: indeed, the maximal ideal $\idealm$ generated by $x-x_0$ and $y-y_0$ is stable under $D$ if and only if $a,b\in \idealm$. Conversely, if $D(f)=\lambda f$, $D(g)=\lambda g$ and $(x_0,y_0)$ is a common zero of $a$ and $b$, then $(x_0,y_0)$ is a common zero of $f$ and $g$ or $(x_0,y_0)$ is a zero of $\lambda$.

\begin{lem}\label{lem_rat}
Let $f/g$ and $f_1/g_1$ be two minimal rational first integrals for $D$. Then there is a $\K$-automorpohism $\tau:\K(x)\to\K(x)$ such that $f/g=\tau(f_1/g_1)$.
\end{lem}

\begin{proof}
If  $(x_0,y_0)\not\in\sing(D)$, then there is a unique ``formal'' solution $\theta:\pol\to\K[[t]]$ of $D$ passing through the maximal ideal $\idealm$ generated by $x-x_0, y-y_0$, that is, $\theta$ is a $\K$-homomorphism such that $\theta D= \partial_t\theta$ with $\idealm=\theta^{-1}((t))$ (see \cite[Thm. 1.1]{BLL} or \cite[Thm. 7]{BP} for more details). Note that  $\idealp=\theta^{-1}((0))$ is the biggest prime ideal contained in $\idealm$ which is stable under $D$. 

  If $g(x_0,y_0)\neq 0$, then there exists a unique  $\alpha=\alpha_{(x_0,y_0)}\in\K$ such that $f+\alpha g$ vanishes at $(x_0,y_0)$. Since  $f+\alpha g$ is an eigenvector we know that all of its irreducible factors define $D$-stable prime divisors. We deduce that one of these factors generate $\idealp$.

  Since $f/g$ is minimal, for a general $\alpha$ the polynomial $f+\alpha g$ is irreducible (Lemma \ref{lem_irr}). Hence there is a Zariski dense open subset $U$ of $\K^2$ such that $(x_0,y_0)\in U$ implies the solution of $D$ passing through that point determines a unique linear combination of $f$ and $g$ over $\K$. By repeating the same reasoning with $f_1/g_1$ we finally conclude $\{f,g\}$ and $\{f_1,g_1\}$ generate the same vector space over $\K$. Thus there are $\alpha,\beta,\gamma,\delta\in\K$, with $\alpha\delta-\beta\gamma\neq 0$ such that $f/g=\frac{\alpha f_1+\beta g_1}{\gamma f_1+\delta g_1}$ which proves the lemma.
  \end{proof}

  Let us now take $\rho\in\aut(D)$.  If $\idealm$ is a maximal ideal in $\sing(D)$, then the maximal ideal $\rho(\idealm)$ satisfies $D(\rho(\idealm))=\rho D(\idealm)\subset \rho(\idealm)$. Then $\rho$ permutes the singular points of $D$.

  On the other hand, $\rho(f/g)$ is clearly another rational first integral, hence Lemma \ref{lem_rat} gives $\rho(f/g)=\tau(f/g)$ for an (unique) automorphism $\tau$ of $\K(x)$. In particular, $\rho$ permutes the maximal ideals associated to $\ideala_{f,g}$.

 As before $f/g$ and $\rho$ may be thought of as a rational map $H:\P^2\tor\P^1$ and a birational map $\rho:\P^2\tor\P^2$, respectively. Hence $H\rho=\tau H$.

 If $\sigma:X\to\P^2$ is the blow-up of the (proper and infinitely near) points where $H$ is not defined, then $H_1:=H\sigma:X\to\P^1$ is a morphism and $\rho_1:=\sigma\rho\sigma^{-1}:X\tor X$ is a birational map such that $\tau H_1=H_1\rho_1$.
  
    \noindent{\bf Proof of Theorem \ref{thm_D}}. Let $H_1:X\to\P^1$ be as above. There is $g\geq 0$ such that a general fiber of that morphism is a smooth irreducible curve o genus $g\geq 0$. By construction such a general fiber is the desingularization of a curve of the form $\alpha f+g$, for a general $\alpha\in\K$, so the first part of the theorem is clear. In order to prove the second part we assume $g\geq 1$.

Now we repeat the reasoning we have already used to prove the first assertion of Theorem \ref{thm_autpolh}. The rational map $H:\P^2\tor\P^1$ above defines a pencil $\Lambda$ on $\P^2$ whose general member is a curve of positive genus. Hence there is not a birational map between natural completions $\varphi:(\P^2,L_\infty)\tor (X,B)$ such that $\varphi_*\Lambda$ is composed by lines. Corollary \ref{cor_blst4} implies $\bir((\P^2,L_\infty),\Lambda)$ is an algebraic group. Moreover, by taking into account \cite[Lem. 2.6]{BlSt} and its proof we conclude that $\aut(D)$ may be embeded in that group as a closed subgroup. Thus $\aut(D)$ is an algebraic group, and this completes the proof of Theorem \ref{thm_D}.     
\qed

\begin{appendix}
\section{Darboux's Theorem}
All along this appendix $\K$ denotes a field of characteristic 0. A derivation  $D\in\derpol$ extends to a derivation $\fra\to\fra$ by means of the quotient formula
\[D(f/g)=\frac{D(f)g-fD(g)}{g^2}.\]  
In particular we have the partial derivatives $\partial_x,\partial_y:\fra\to\fra$. A straightforward calculation gives that the Schwartz equality 
\[\partial_x\circ\partial_y=\partial_y\circ\partial_x\]
holds on $\fra$; we will also write $f_x=\partial_x(f)$ and $f_y=\partial_y(f)$.

We have a ``dual de Rham complex'' (we are identifying $\derpol\wedge\derpol$ with $\pol$ in the natural form)
 \[\xymatrix{0\ar@{->}[r]&\pol\ar@{->}[r]^{\delta_1\ \ \ }&\derpol\ar@{->}[r]^{\ \ \ \delta_2}&\pol\ar@{->}[r]&0},\]
where $\delta_1(f)=f_y\partial_x-f_x\partial_y$ and $\delta_2(g\partial_x+h\partial_y)=g_x+h_y$. 

Finally, we consider the homomorphism of $\pol$-modules $\wedge D:\derpol\to\pol$, $D'\mapsto D'\wedge D$. If $D=a\partial_x+b\partial_y$ we denote by $I_D$ the ideal generated by $a$ and $b$ and take $c\in\pol$ to be a higher common divisor of $a$ and $b$. We have an exact sequence
\[\xymatrix{0\ar@{->}[r]&\pol D_0\ar@{->}[r]&{\derpol}\ar@{->}[r]^{\ \ \ \wedge D}&I_D\ar@{->}[r]&0},\] 
where $D_0=(a/c)\partial_x+(b/c)\partial_y$. 

\begin{rems}\label{rem_eigenvector}

a) The composition $\delta_1\circ \wedge D$ is nothing but $D$, and moreover, we have $\delta_2(fD)=D(f)+f\delta_2(D)$.
 
b) By considering $D$ as a rational $\K$-derivation we obtain a dual de Rham complex of vector spaces over $\fra$ and an exact sequence as above where we replace $D_0$ and $I_D$ with $D$ and $\fra$, respectively.
\end{rems}

The following Lemma is straightforward 

\begin{lem}\label{lem_darboux}
Let  $D\in\derpol$ be a derivation and let $g_1,\ldots, g_\ell\in\pol$ be coprime polynomials such that $g_i\not\in\ker D$ for all $i$; set $q_i:=g_1\cdots\hat{g_i}\cdots g_\ell$, $i=1,\ldots,\ell$. If $g_i,D(g_i)$ are coprime for every $i$, then the  polynomials 
 \[ q_iD(g_i), i=1,\ldots,\ell,\]
are linearly independent over $\K$.
\end{lem}
\qed

Now we are prepared to prove 

\begin{thm}$($Darboux, 1879$)$
Let $D=a\partial_x+b\partial_y$ be a derivation and set $d:=\max\{\deg a,\deg b\}$. If $\eig(D)$ contains at least $2+d(d+1)/2$ irreducible elements, then there is $h\in\fra\setminus\K$ such that $D(h)=0$. 
\end{thm}
\begin{proof}
  Let $\{f_1,\ldots,f_N\}\subset\eiga(D)$ be a subset of irreducible eigenvector such that $[f_i]\neq [f_j]$ for if $i\neq j$, and assume $N\geq 2+d(d+1)/2$. Let $\lambda_1,\ldots,\lambda_N$ be eigenvalues of $f_1,\ldots,f_N$, respectively. If some $\lambda_i$ is zero, there is nothing to prove, then we assume also $\lambda_i\neq 0$ for all $i=1,\ldots,N$. Hence $a\partial_x(f_i)+b\partial_y(f_i)$ is different from zero and is divided by $f_i$, for all $i$. We deduce $\deg\lambda_i\leq d-1$ for  $i=1,\ldots,N$.

Since the space of polynomials of degree $\leq d-1$ has dimension $d(d+1)/2$ we deduce there are $a_1,\ldots,a_{N-1},b_2,\ldots,b_N\in\K$ such that 
\begin{equation}\label{eq_lincomb}
\sum_{i=1}^{N-1}\alpha_i\lambda_i=\sum_{j=2}^N\beta_j\lambda_j=0;
\end{equation}  by reordering, if necessary, we suppose $a_1b_N\neq 0$.  

Now we consider the rational derivations 
\[D_1=\sum_{i=1}^{N-1}\frac{\alpha_i}{f_i}(\delta_1(f_i)),\ \ D_2=\sum_{j=2}^N\frac{\beta_j}{f_j}(\delta_1(f_j)).\]
Since $\delta_2(\delta_1(f_i)/f_i)=0$ for all $i$ we get $\delta_2(D_1)=\delta_2(D_2)=0$. Note that since $\alpha_1\beta_N\neq 0$ Lemma \ref{lem_darboux} implies $D_1\neq 0$ and $D_2\neq 0$ (recall $\car(\K)=0$), and a posteriori both derivations are linearly independent over $\K$.

By Remark \ref{rem_eigenvector}a) we know the equalities in (\ref{eq_lincomb}) signify $D_1,D_2\in\ker(\wedge D)$. Hence there is $h\in\fra\setminus\K$ such that $D_1=h D_2$. By applying that remark again we also deduce $\delta_1(h)=0$, which completes the proof.
\end{proof}

\end{appendix}

\end{document}